\documentclass[10pt]{article}
\usepackage[latin1]{inputenc}
\usepackage{epsfig}
\usepackage{color}
\usepackage[british,english]{babel}
\usepackage{amsthm}
\usepackage{amsmath}
\usepackage{amsfonts}
\usepackage{amssymb}

\usepackage{mathtools}
\usepackage{esint}

\usepackage{graphicx}
\numberwithin{equation}{section}

\newtheorem{corollary}{Corollary}[section]
\newtheorem{definition}[corollary]{Definition}

\newtheorem{lemma}[corollary]{Lemma}

\newtheorem{remark}[corollary]{Remark}
\newtheorem{theorem}[corollary]{Theorem}
\newfont{\sBlackboard}{msbm10 scaled 900}

\newcommand{\mylabel}[1]{\label{#1}
            \ifx\undefined\stillediting
            \else \fbox{$#1$}\fi }
\newcommand{\BE}{\begin{equation}}

\newcommand{\EEQ}{\end{equation}}
\newcommand{\rfb}[1]{\mbox{\rm
   (\ref{#1})}\ifx\undefined\stillediting\else:\fbox{$#1$}\fi}

\newfont{\Blackboard}{msbm10 scaled 1200}

\newfont{\roma}{cmr10 scaled 1200}

\def\CC{\rm \hbox{C\kern-.56em\raise.4ex
         \hbox{$\scriptscriptstyle |$}\kern+0.5 em }}

\def\n{|\kern -.05cm{|}\kern -.05cm{|}}

\def\R{{ \hbox{\sc I\hskip -2pt R}}} 
\def\Z{{ Z}} 

\def \noame{\noalign{\medskip}}
\newcommand{\mm}    {{\hbox{\hskip 0.5pt}}}

\newcommand{\bluff} {{\hbox{\raise 15pt \hbox{\mm}}}}

\newcommand{\ep}   {\varepsilon}
\newcommand{\eps}   {\ep}

\usepackage{fancyhdr}

\makeatletter
\def\section{\@startsection {section}{1}{\z@}{-3.5ex plus -1ex minus
    -.2ex}{2.3ex plus .2ex}{\bf \large}}
\makeatother
\def\be{\begin{equation}}
\def\ee{\end{equation}}

\date{ }
\begin{document}
\thispagestyle{empty}
\title{\Large    Modeling Carreau fluid flows through a very thin porous medium}\maketitle
\vspace{-2cm}
\begin{center}
Mar\'ia ANGUIANO\footnote{Departamento de An\'alisis Matem\'atico. Facultad de Matem\'aticas. Universidad de Sevilla. 41012-Sevilla (Spain) anguiano@us.es}, Matthieu BONNIVARD\footnote{Centrale Lyon, CNRS, INSA Lyon, Universite Claude Bernard Lyon 1, Université Jean Monnet, ICJ UMR5208, 69130 Ecully,
	France.} and Francisco Javier SU\'AREZ-GRAU\footnote{Departamento de Ecuaciones Diferenciales y An\'alisis Num\'erico. Facultad de Matem\'aticas. Universidad de Sevilla. 41012-Sevilla (Spain) fjsgrau@us.es}
 \end{center}

 \renewcommand{\abstractname} { Abstract}
\begin{abstract}
  This study investigates three-dimensional, steady-state, and non-Newtonian flows within a very thin porous medium (VTPM). The medium is modeled as a domain confined between two parallel plates and perforated by solid cylinders that connect the plates and are distributed periodically in perpendicular directions. We denote the order of magnitude of the thickness of the domain by $\eps$ and  define the period and order of magnitude of the cylinders' diameter by $\eps^\ell$, where $0<\ell<1$ is fixed. In other words, we consider the regime $\eps\ll\ep^\ell$. We assume that the viscosity of the non-Newtonian fluid follows Carreau's law and is scaled by a factor of $\eps^\gamma$, where $\gamma$ is a real number.  Using asymptotic techniques with respect to the thickness of the domain, we perform a new, complete study of the asymptotic behaviour of the fluid as $\eps$ tends to zero. Our mathematical analysis is based on deriving sharp a priori estimates through pressure decomposition, and on compactness results for the rescaled velocity and pressure, obtained using the unfolding method. Depending on $\gamma$ and the flow index $r$,  we rigorously derive different linear and nonlinear reduced limit systems. These systems allow us to obtain explicit expressions for the filtration velocity and simpler Darcy's laws for limit pressure.
  
  \end{abstract}
\bigskip\noindent

\noindent {\small  AMS classification numbers: 35B37, 76M50.}  \\

\ \\
\section {Introduction}\label{S1}

This paper discusses the asymptotic behaviour of the flow of a non-Newtonian fluid, whose viscosity follows Carreau's law (see Definition~\eqref{Carreaulaw}), through a very thin  porous medium. This terminology was proposed by Fabricius et al.\!~\cite{Fabricius} in the context of an incompressible Newtonian fluid flow through a thin porous medium consisting of a domain confined between two parallel plates and perforated by solid cylinders that connect the plates and are periodically distributed in perpendicular directions. If we denote by $\eps$ the order of magnitude of the thickness of the domain, and by $\ep^\ell$ (where $\ell>0$ is fixed) the period and the order of magnitude of the cylinders' diameter, three different regimes were identified in~\cite{Fabricius}, depending on the relation between the height of the domain and the size of the rigid inclusions.

\begin{itemize}
	
	\item {\bf Homogeneously thin porous media (HTPM)} correspond to the case where the cylinder height is much larger than the interspatial distance, i.e., $\ep^\ell \ll \ep$, which is equivalent to $\ell>1$.

	\item {\bf Very thin porous media (VTPM)}  correspond to the opposite case where the cylinder height is much smaller than the interspatial distance, i.e., $\ep^\ell \gg \ep$, or equivalently, $0<\ell<1$.

	\item {\bf Proportionally thin porous media (PTPM)} correspond to the critical case where the cylinder height is proportional to the interspatial distance, i.e., $\ep^\ell= \ep$ and $\ell=1$.

\end{itemize}

In the case of Newtonian fluid flow described by Stokes or stationary Navier-Stokes equations, all regimes were analyzed in~\cite{Fabricius}, where the authors proved that the flow is governed by a two-dimensional Darcy equation, different for each regime, using the asymptotic expansion method. A more rigorous approach was later proposed by Anguiano and Su\'arez-Grau~\cite{Anguiano_SG_Transition}, based on the unfolding method introduced by Cioranescu et al.\!~\cite{Ciora2,Cioran-book}. The studies mentioned above are in line with a large number of contributions concerning the derivation of Darcy's law from hydrodynamic equations for Newtonian flows through periodic porous media. We refer to~\cite{Allaire, Lipton, SanchezPalencia, Tartar} for derivations using homogenization, and to the book by Hornung~\cite{HornungBook} for various physical aspects and mathematical results around this topic.

 In the case of non-Newtonian fluids, the derivation of Darcy's laws in limit models is even more challenging because viscosity is now a nonlinear function of the symmetrized gradient of the velocity field. In the case where the viscosity follows a power law, we refer to Anguiano and Su\'arez-Grau \cite{Anguiano_SG}, where the same critical regime is identified and the three regimes described above are analyzed, leading to nonlinear 2D Darcy models in the limit. In the case of a quasi-Newtonian fluid whose viscosity is given by the Carreau law, Bourgeat and Mikeli\'c~\cite{Bourgeat1} (see also~\cite{Bourgeat2} and~\cite{Kalousek}) used the two-scale convergence method to derive the averaged law describing the flow in a periodic porous medium of fixed height. In the case of a thin porous medium, a complete study of the regime PTPM ($\ell=1$) has been recently conducted by the authors in~\cite{Anguiano_Bonn_SG_Carreau1, Anguiano_Bonn_SG_Carreau2}, where the viscosity is scaled by a factor of $\ep^\gamma$, $\gamma\in \mathbb{R}$.

Now, let us define the VTPM model and the associated notation precisely before stating the equations of motion for the fluid.

\subsection{Definition of the very thin porous medium}\label{sec:definition}

Let $\eps$ be a sequence of positive numbers that converges to zero. The geometry of a very thin porous medium is characterized by the presence of two different microscales:
\begin{itemize}
	\item the scale $\ep$ related to the film thickness,
	\item the scale $\ep^\ell$ related to the diameter of the cylindrical obstacles, and their periodic distribution in the horizontal directions.
\end{itemize}
In the case of a VTPM, we assume that $\ep^\ell$ is of order greater than $\ep$, i.e., $0<\ell<1$, which implies that
\begin{equation}\label{RelationAlpha}
\lim_{\ep\to 0}\ep^{1-\ell}=0.
\end{equation}

We denote by $\Omega_\eps$ the porous medium, that takes the form
\begin{equation}\label{Def:Omegaeps}
\Omega_\ep := \omega_\ep\times (0,\eps)
\end{equation}
where $\omega_\ep$ is a bounded, connected open subset of $\R^2$, associated to a microstructure at scale $\ep^\ell$.

The microstructure is described by a periodic cell $Z^{\prime}=(-1/2,1/2)^2$, which is made of two complementary parts: the fluid part $Z^{\prime}_{f}$, and the solid part $T^{\prime}$ ($Z^{\prime}_f  \cup \overline{T^{\prime}}=Z^\prime$ and $Z^{\prime}_f  \cap T^{\prime}=\emptyset$), see Figure \ref{fig:cell}. We assume that $T^{\prime}$ is an open connected subset of $Z^\prime$, such that $\overline T^\prime$ is strictly included  in $Z^\prime$. We also assume that $T^\prime$ is a Lipschitz domain, with uniform $C^2$ regularity in the sense of Definition~\ref{Def:C2domain}.

We introduce a bounded, connected open set $\omega\subset\R^2$, that we also assume to be uniformly $C^2$. The exterior normal to $\partial\omega$ is denoted by $n$.
 The domain $\omega$ is covered by a regular mesh of squares of size ${\varepsilon^\ell}$: for $k^{\prime}\in \mathbb{Z}^2$, each cell $Z^{\prime}_{k^{\prime},{\varepsilon^\ell}}={\varepsilon^\ell}k^{\prime}+{\varepsilon^\ell}Z^{\prime}$ is divided into a fluid part $Z^{\prime}_{f_{k^{\prime}},{\varepsilon^\ell}}$ and a solid part $T^{\prime}_{k^{\prime},{\varepsilon^\ell}}$, i.e., is similar to the unit cell $Z^{\prime}$ rescaled to size ${\varepsilon}^\ell$.

 We set $Z=Z^{\prime}\times (0,1)\subset \mathbb{R}^3$, which is divided into a fluid part $Z_{f}=Z'_f\times (0,1)$ and solid part $T=T'\times(0,1)$. Consequently $Z_{k^{\prime},{\varepsilon^\ell}}=Z^{\prime}_{k^{\prime},{\varepsilon^\ell}}\times (0,1)\subset \mathbb{R}^3$  is also divided into a fluid part $Z_{f_{k^{\prime}},{\varepsilon^\ell}}$ and a solid part $T_{{k^{\prime}},{\varepsilon^\ell}}$, see Figures \ref{fig:cell} and \ref{fig:cellep}.

We denote by $\tau(\overline T'_{k',\varepsilon^\ell})$ the set of all translated images of $\overline T'_{k',\varepsilon^\ell}$, for $k'\in \Z^2$. The set $\tau(\overline T'_{k',\varepsilon^\ell})$ represents the obstacles projected in $\mathbb{R}^2$.

Based on Definition~\eqref{Def:Omegaeps}, the bottom $\omega_{\ep}$ of the porous medium is defined as
\[
\omega_{\ep}=\omega\backslash\bigcup_{k^{\prime}\in \mathcal{K}_{\varepsilon}} \overline T^{\prime}_{{k^{\prime}},{\varepsilon^\ell}},
\] 
where $\mathcal{K}_{\varepsilon}=\{k^{\prime}\in \mathbb{Z}^2: Z^{\prime}_{k^{\prime}, {\varepsilon^\ell}} \cap \omega \neq \emptyset \}$. The associated domain $\Omega_\eps$ is shown in Fig.~\ref{fig:omep}.

We assume that for every $k'\in \mathcal K_\eps$, the obstacle $\overline T'_{k',\varepsilon^\ell}$ is at a distance of order $O(\ep^\ell)$ from the boundary $\partial\omega$. More specifically, we assume the following:
\begin{equation}\label{Assumption:distanceBoundary}
\forall \eps>0,	\forall k'\in \mathcal K_\ep\quad \mathrm{dist}(\overline T'_{k',\varepsilon^\ell},\partial\omega) \geq \ep^\ell k_0.
\end{equation}
In the above statement, constant $k_0>0$ is introduced in the proof of Lemma~\ref{Lemma:extensionforpressure}. This is a technical assumption that guarantees the existence of the extension operator $E_\eps$ introduced in the previously mentioned lemma.

We also introduce the rescaled porous domain $\widetilde{\Omega}_\ep$, limit domain $\Omega$ and thin layer $Q_\eps$ defined by
\begin{equation}\label{OmegaTilde}
\widetilde{\Omega}_{\varepsilon}=\omega_{\varepsilon}\times (0,1), \quad \Omega=\omega\times (0,1), \quad Q_\varepsilon=\omega\times (0,\varepsilon).
\end{equation}
We observe that $\widetilde{\Omega}_{\varepsilon}=\Omega\backslash \bigcup_{k^{\prime}\in \mathcal{K}_{\varepsilon}} \overline T_{{k^{\prime}}, {\varepsilon^\ell}},$ and we define $T_\varepsilon=\bigcup_{k^{\prime}\in \mathcal{K}_{\varepsilon}} T_{k^\prime, \varepsilon^\ell}$ as the set of the solid cylinders contained in $\widetilde \Omega_\varepsilon$. 

We finally consider the top and bottom boundaries
\begin{equation}\label{Def:top_bottom_boundaries}
\Gamma_0=\omega\times \{0\}, \quad \Gamma_1=\omega\times \{1\},\quad \widehat \Gamma_0=  Z'\times \{0\},\quad \widehat \Gamma_1=  Z'\times \{1\}.
\end{equation}

\begin{figure}[h!]
\begin{center}
\includegraphics[width=4cm]{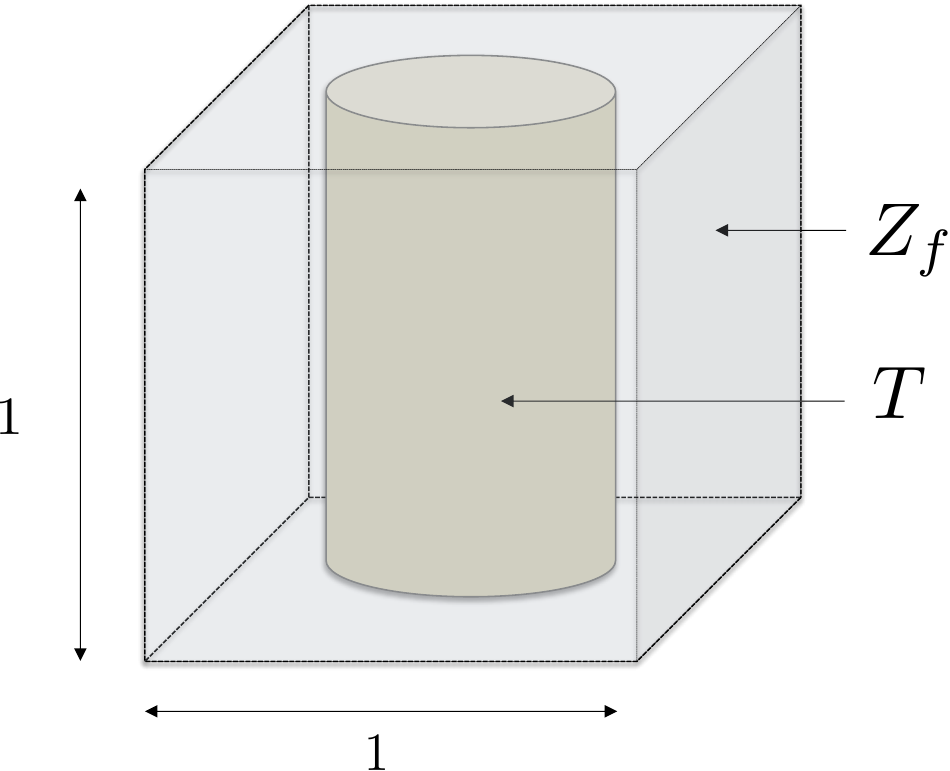}
\hspace{2.5cm}
\raisebox{.1\height}{\includegraphics[width=3cm]{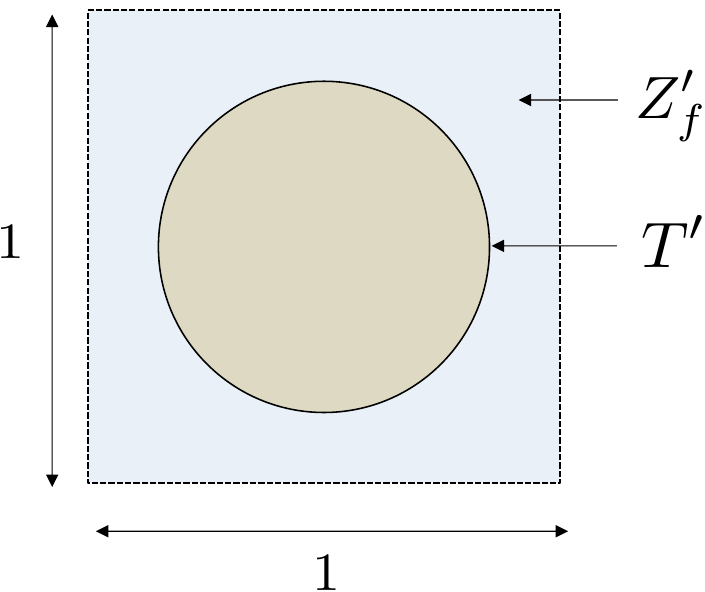}}
\end{center}
\vspace{-0.4cm}
\caption{View of the 3D reference cell  $Z$ (left) and the 2D reference cell $Z'$ (right).}
\label{fig:cell}
\end{figure}

\begin{figure}[h!]
\begin{center}
\includegraphics[width=5cm]{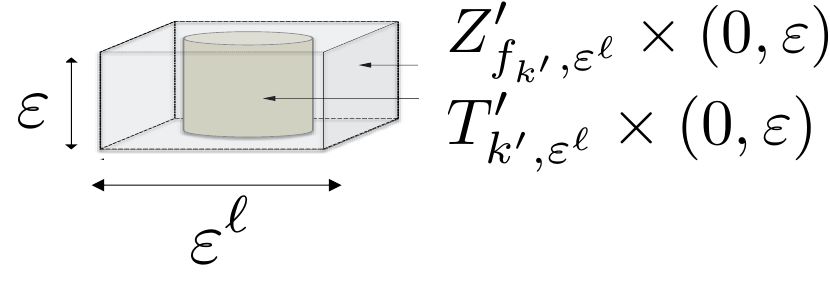}
\hspace{2.5cm}
\raisebox{.1\height}{\includegraphics[width=3cm]{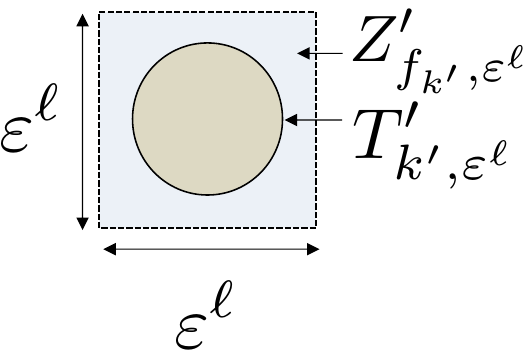}}
\end{center}
\vspace{-0.4cm}
\caption{View of the 3D reference cell  $Z'_{k',\varepsilon^\ell}\times (0,\varepsilon)$ (left) and the 2D reference cell $Z'_{k',\varepsilon^\ell}$ (right).}
\label{fig:cellep}
\end{figure}

\begin{figure}[h!]
\begin{center}
\includegraphics[width=12cm]{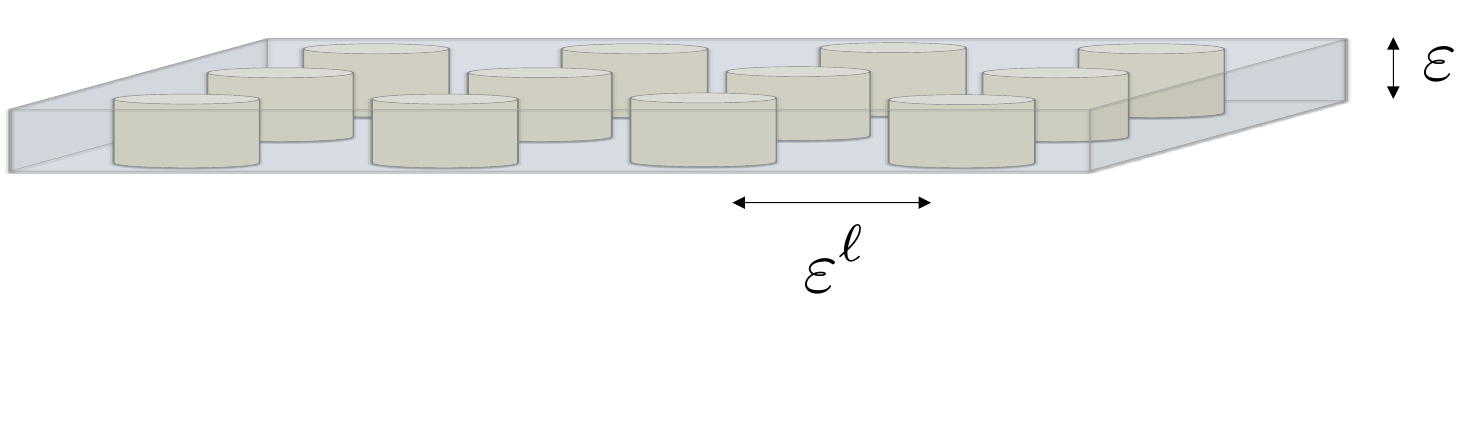}
\end{center}
\vspace{-1.4cm}
\caption{View of the thin porous media $\Omega_\varepsilon$.}
\label{fig:omep}
\end{figure}

\subsection{Additional notations}

In this section, we present some additional notations that will be useful throughout this paper.

The points $x\in\mathbb{R}^3$ are decomposed as $x=(x^{\prime},x_3)$ with $x^{\prime}=(x_1,x_2)\in \mathbb{R}^2$, $x_3\in \mathbb{R}$. We also use the notation $x^{\prime}$ to denote a generic vector of $\mathbb{R}^2$.

We note $B'(x',r)$ the open ball of center $x'$ and radius $r>0$ in $\R^2$.

The symbol $:$ denotes the full contraction of two matrices, that is, for $A=(a_{ij})_{1\leq i,j\leq 3}$ and $B=(a_{ij})_{1\leq i,j\leq 3}$, we have $A:B=\sum_{i,j=1}^3a_{ij}b_{ij}$.

\paragraph{Differential operators}

Let us consider a vectorial function ${  v} =( {  v} ', v_{3})$ with ${  v}'=(v_1, v_2)$ and a scalar function $\varphi$ , both defined in $\Omega_\varepsilon$. Then, we introduce operators $\Delta$ and $\nabla$ defined by
$$\begin{array}{c}
 \displaystyle \Delta{  v}=\Delta_{x'}{  v} +\partial_{x_3}^2 {  v},\quad {\rm div}({  v} )={\rm div}_{x'}({  v}')+\partial_{x_3}v_3,\quad \nabla \varphi=(\nabla_{x'} \varphi,  \partial_{x_3}\varphi)^t,\\
 \end{array}$$
and
$\mathbb{D}:\mathbb{R}^3\to \mathbb{R}^3_{\rm sym}$   the symmetric part of the velocity gradient, that is
$$\mathbb{D}[{  v}]={1\over 2}(D{  v}+(D{  v})^t)=\left(\begin{array}{ccc}
\partial_{x_1}v_1 &   {1\over 2}(\partial_{x_1}v_2 + \partial_{x_2}v_1) &   {1\over 2}(\partial_{x_3}v_1 + \partial_{x_1}v_3)\\
\noame
 {1\over 2}(\partial_{x_1}v_2 + \partial_{x_2}v_1) & \partial_{x_2}v_2 &   {1\over 2}(\partial_{x_3}v_2 + \partial_{x_2}v_3)\\
 \noame
 {1\over 2}(\partial_{x_3}v_1 + \partial_{x_1}v_3)&   {1\over 2}(\partial_{x_3}v_2 + \partial_{x_2}v_3)& \partial_{x_3}v_3
\end{array}\right).$$

\noindent Moreover, for a vectorial function $\widetilde {   v} =(\widetilde {  v}', \widetilde  v_{3})$ and a scalar function $\widetilde \varphi$, both defined in $\widetilde\Omega_\ep$, obtained from ${ v}$ and $\varphi$ after a dilatation in the vertical variable (i.e., $z_3=x_3/\ep$), we will use the following operators:
 $$\begin{array}{c}
 \displaystyle \Delta_{\varepsilon}  \widetilde  {  v} =\Delta_{x'} \widetilde  {  v} + \varepsilon^{-2}\partial_{z_3}^2 \widetilde  {  v} ,\quad   \displaystyle\Delta_{ \varepsilon}\widetilde \varphi=\Delta_{x'}\widetilde \varphi+ \ep^{-2}\partial^2_{z_3}\widetilde \varphi,\\
  \noame
  (D_{\ep}\widetilde {  v})_{ij}=\partial_{x_j}\widetilde {v}_i\ \hbox{ for }\ i=1,2,3,\ j=1,2,\quad   (D_{\ep}\widetilde {  v})_{i3}=\ep^{-1}\partial_{z_3}\widetilde {v}_i \ \hbox{ for }\ i=1,2,3,\\
  \noame
  \nabla_{\ep}\widetilde\varphi=(\nabla_{x'}\widetilde \varphi, \ep^{-1}\partial_{z_3}\widetilde\varphi)^t,\quad {\rm div}_{\varepsilon}(\widetilde  {    v})={\rm div}_{x'}(\widetilde {  v}')+\ep^{-1}{\partial_{z_3}}\widetilde  v_{3},
  \end{array}$$
 Moreover, we define $\mathbb{D}_{\ep}[\widetilde  { v}]$ as follows
 $$\mathbb{D}_{\ep}[\widetilde  { v}]=\mathbb{D}_{x'}[\widetilde  { v}]+\ep^{-1}\partial_{z_3}[\widetilde  { v}]=\left(\begin{array}{ccc}
\partial_{x_1}\widetilde  v_1 &   {1\over 2}(\partial_{x_1}\widetilde  v_2 + \partial_{x_2}\widetilde  v_1) &   {1\over 2} (\partial_{x_1}\widetilde  v_3+_\ep^{-1}\partial_{z_3}\widetilde  v_1)\\
\noame
 {1\over 2}(\partial_{x_1}\widetilde  v_2 + \partial_{x_2}\widetilde  v_1) & \partial_{x_2}\widetilde  v_2 &   {1\over 2} (\partial_{x_2}\widetilde  v_3+_\ep^{-1}\partial_{z_3}\widetilde  v_2)\\
 \noame
 {1\over 2} (\partial_{x_1}\widetilde  v_3+\ep^{-1}\partial_{z_3}\widetilde  v_1)&   {1\over 2} (\partial_{x_2}\widetilde  v_3+\ep^{-1}\partial_{z_3}\widetilde  v_2)& \ep^{-1}\partial_{z_3}\widetilde  v_3
\end{array}\right),$$
 where $\mathbb{D}_{x'}[\widetilde  { v}]$ and $\partial_{z_3}[\widetilde  { v}]$ are defined by
\begin{equation}\label{def_der_sym_1}
  \mathbb{D}_{x'}[{  v}]=\left(\begin{array}{ccc}
\partial_{x_1}v_1 &   {1\over 2}(\partial_{x_1}v_2 + \partial_{x_2}v_1) &   {1\over 2} \partial_{x_1}v_3\\
\noame
 {1\over 2}(\partial_{x_1}v_2 + \partial_{x_2}v_1) & \partial_{x_2}v_2 &   {1\over 2} \partial_{x_2}v_3\\
 \noame
 {1\over 2} \partial_{x_1}v_3&   {1\over 2} \partial_{x_2}v_3& 0
\end{array}\right),
 \  \partial_{z_3}[{ v}]=\left(\begin{array}{ccc}
0 &   0&   {1\over 2} \partial_{z_3}v_1\\
\noame
 0& 0 &   {1\over 2} \partial_{z_3}v_2\\
 \noame
 {1\over 2} \partial_{z_3}v_1&   {1\over 2} \partial_{z_3}v_2& \partial_{z_3}v_3
\end{array}\right).
\end{equation}

\noindent We also define the following operators applied to ${ v}'$:
\begin{equation}\label{def_der_sym_2}
  \mathbb{D}_{x'}[{  v}']=\left(\begin{array}{ccc}
\partial_{x_1}v_1 &   {1\over 2}(\partial_{x_1}v_2 + \partial_{x_2}v_1) &   0\\
\noame
 {1\over 2}(\partial_{x_1}v_2 + \partial_{x_2}v_1) & \partial_{x_2}v_2 &  0\\
 \noame
0&   0& 0
\end{array}\right),
 \quad \partial_{z_3}[{  v}']=\left(\begin{array}{ccc}
0 &   0&   {1\over 2} \partial_{z_3}v_1\\
\noame
 0& 0 &   {1\over 2} \partial_{z_3}v_2\\
 \noame
 {1\over 2} \partial_{z_3}v_1&   {1\over 2} \partial_{z_3}v_2& 0
\end{array}\right).
\end{equation}
We note that for vectorial functions $v$ and $w$, according to previous definitions, we have
\begin{equation}\label{productAB}
\mathbb{D}_{x'}[v]:\mathbb{D}_{x'}[w']=\mathbb{D}_{x'}[v']:\mathbb{D}_{x'}[w'],\quad \partial_{z_3}[v]:\partial_{z_3}[w']=\partial_{z_3}[v']:\partial_{z_3}[w'].
\end{equation}
Finally, we denote by $O_\ep$ a generic real sequence that tends to zero with $\ep$ and can change from line to line, and by $C$ a generic positive constant that can also change from line to line.

\paragraph{Functional spaces}

We define the following sets for $1<q<+\infty$. Let $C^\infty_{\#}(Z)$ be the space of infinitely differentiable functions in $\mathbb{R}^3$ that are $Z'$-periodic. By $L^q_{\#}(Z_f)$,  we denote its completion in the norm $L^q(Z_f)$ and by $L^q_{0,\#}(Z_f)$  the space of functions in $L^q_{\#}(Z_f)$ with mean value zero. Moreover, we introduce
\begin{equation}\label{Vz3}
	\begin{array}{c}
		\displaystyle
		V_{z_3}^q(\Omega)=\{\varphi\in L^q(\Omega)\ :\ \partial_{z_3}\varphi\in L^q(\Omega)\},\quad V_{z_3,\#}^q(Z_f)=\{\varphi\in L^q_\#(Z_f)\ :\ \partial_{z_3}\varphi\in L^q(Z_f)\},\\
		\noame
		\displaystyle V_{z_3,\#}^q(\omega\times Z_f)=\{\varphi\in L^q(\omega;L^q_\#(Z_f))\ :\ \partial_{z_3}\varphi\in L^q(\omega\times Z_f)\}.
	\end{array}
\end{equation}

\subsection{Model problem}\label{sec:model} With homogeneous Dirichlet boundary conditions, the flow of velocity ${ u}_\ep=({ u}'_\ep(x), u_{3,\ep}(x))$ and pressure $p_\ep=p_\ep(x),$ at a point $x\in \Omega_\ep$, is supposed to be ruled by the Stokes system
\begin{equation}\label{system_1}
\left\{\begin{array}{rl}
\displaystyle -\varepsilon^\gamma{\rm div}(\eta_r(\mathbb{D}[{  u}_\ep])\mathbb{D}[{  u}_\ep])+\nabla p_\ep={  f} & \hbox{in}\ \Omega_\varepsilon,\\
\noame
{\rm div}( {  u}_\varepsilon)=0& \hbox{in}\ \Omega_\varepsilon\\
\noame
\displaystyle {  u}_\varepsilon=0 & \hbox{on }\partial Q_\varepsilon\cup \partial S_\ep.
\end{array}\right.
\end{equation}
Here, the viscosity $\eta_r$ follows the Carreau law defined by
\begin{equation}\label{Carreaulaw}
\eta_r(\mathbb{D}[u])=(\eta_0-\eta_\infty)(1+\lambda|\mathbb{D}[u]|^2)^{{r\over 2}-1} + \eta_\infty,\quad 1<r<+\infty,\ r\neq 2,\quad \eta_0>\eta_\infty>0,\quad \lambda>0,
\end{equation}
which is scaled by a factor of $\eps^\gamma$, where $\gamma\in\mathbb{R}$ in equation~\eqref{system_1}. If $1<r<2$, the fluid is pseudoplastic; if $r>2$, it is dilatant. 

As is usual in the study of viscous fluid flows in thin domains, we assume the source term ${  f}$  to be of the form
\begin{equation}\label{fassump}
{  f} (x)=({ f}'(x'),0)\quad \hbox{with}\quad {  f}'\in L^\infty(\omega)^2.
\end{equation}
 
\begin{remark} Under previous assumptions, for every fixed positive $\varepsilon$, the classical theory (see for instance \cite{Tapiero2,Bourgeat1,Lions2}), ensures the existence of a unique weak  solution $(u_\varepsilon, p_\varepsilon)\in H^1_0(\Omega_\varepsilon)^3\times L^{2}_0(\Omega_\varepsilon)$, for $1<r< 2$, and $(u_\varepsilon, p_\varepsilon)\in W^{1,r}_0(\Omega_\varepsilon)^3\times L^{r'}_0(\Omega_\varepsilon)$ with $1/r+1/r'=1$, for $r> 2$, where  $L^{2}_0$ (respectively $L^{r'}_0$) is the space of functions of $L^2$ (respectively $L^{r'}$) with mean value zero.
\end{remark}
To study the asymptotic behaviour of the solutions ${  u}_\ep$ and $p_\ep$ when $\ep$  tends to zero,  we use the rescaling  
\begin{equation}\label{dilatacion}
z_3={x_3\over \ep}\,,
\end{equation}
 to manipulate functions defined in $\widetilde\Omega_\ep$, which is defined in (\ref{OmegaTilde}) and has a constant height equal to one. Using the change of variables (\ref{dilatacion}) in the model problem, we obtain the rescaled Stokes system:
\begin{equation}\label{system_1_dil}
\left\{\begin{array}{rl}
\displaystyle -\varepsilon^\gamma {\rm div}_\varepsilon(\eta_r(\mathbb{D}_\varepsilon[ \widetilde {  u}_\ep])\mathbb{D}_\varepsilon[\widetilde { u}_\ep])+\nabla_\varepsilon \widetilde p_\ep={ f} & \hbox{in}\ \widetilde \Omega_\varepsilon,\\
\noame
{\rm div}_\varepsilon( \widetilde { u}_\varepsilon)=0& \hbox{in}\ \widetilde \Omega_\varepsilon\\
\noame
\displaystyle \widetilde { u}_\varepsilon=0 & \hbox{on }\partial \Omega \cup \partial T_\ep,
\end{array}\right.
\end{equation}
 where the unknown functions in the above system are given by ${  \widetilde u}_\varepsilon(x',z_3)={  u}_\varepsilon(x',\varepsilon z_3)$, $ \widetilde p_\varepsilon(x',z_3)=p_\varepsilon(x',\varepsilon z_3)$  for almost every $(x',z_3)$ in $\widetilde\Omega_\varepsilon$, and the operators ${\rm div}_\ep$, $\mathbb{D}_\ep$ and $\nabla_\ep$ are defined in Section \ref{sec:definition}.

Our goal is to describe the asymptotic behaviour of this new sequence $(\widetilde{u}_{\varepsilon}, \widetilde{p}_{\varepsilon})$ when $\eps$ tends to zero, depending on the value of $\gamma$ and the flow index $r$. Since this study is rather lengthy and technical, for the sake of clarity, we focus in the next subsection on the statement of our main results.

\subsection{Main results}

Our main results consist of identifying the limit problem satisfied by the filtration velocity $\widetilde V$, defined on $\omega$ by
\begin{equation}\label{Def:filtrationVelocity}
	\widetilde V(x')=\int_0^{1}\widetilde { u}(x',z_3)\,dz_3,
\end{equation}
where $\widetilde { u}$ is the limit introduced in Lemma~\ref{lem_conv_vel}. In Theorem~\ref{expresiones_finales}, we present the results concerning pseudoplastic fluids ($1<r<2$) and in Theorem~\ref{expresiones_finalesrsup2} those concerning dilatant fluids ($r>2$).

\begin{theorem}[Case $1<r<2$]\label{expresiones_finales}
	Depending on the value of $\gamma$, filtration velocity $\widetilde V$ can be expressed as follows:
	\begin{itemize}
		\item If $\gamma=1$,  $\widetilde V$ is given by
		\begin{equation}\label{FiltrationVelcovcity}
			\widetilde V'(x')=-2\int_{Z'}\left(\int_{-{1\over 2}}^{1\over 2}{\left({1\over 2}+\xi\right)\xi\over \psi(2|\nabla_{x'}\widetilde p(x')-{ f}'(x')+\nabla_{z'}\widehat q(z')||\xi|)}d\xi\right)\left(\nabla_{x'}\widetilde p(x')-{f}'(x')+\nabla_{z'}\widehat q(z')\right)dz',
		\end{equation}
		and $\widetilde V_3\equiv 0$ in $\omega$, where $\widehat q(z')\in L^{2}_{0,\#}(Z')\cap H^1(Z')$ is given by $\widehat q_{\delta'}$ with $\delta'=\nabla_{x'}\widetilde p(x')-{ f}'(x')$, which is the unique solution of the local problem 
		\begin{equation}\label{expressionq}
			\left\{\begin{array}{l}
				\displaystyle
				{\rm div}_{z'}\left(\left(\int_{-{1\over 2}}^{1\over 2}{\left({1\over 2}+\xi\right)\xi\over \psi(2|\delta'+\nabla_{z'}\widehat q_{\delta'}(z')||\xi|)}d\xi\right)\left(\delta'+\nabla_{z'}\widehat q_{\delta'}(z')\right)
				\right)=0\quad\hbox{in }Z'_f\\
				\noame
				\displaystyle
				\left(\left(\int_{-{1\over 2}}^{1\over 2}{\left({1\over 2}+\xi\right)\xi\over \psi(2|\delta'+\nabla_{z'}\widehat q_{\delta'}(z')||\xi|)}d\xi\right)\left(\delta'+\nabla_{z'}\widehat q_{\delta'}(z')\right)
				\right)\cdot n=0\quad\hbox{on }\partial T'.
			\end{array}\right.
		\end{equation}
		Here, $\psi$ is the inverse function of 
		\begin{equation}\label{taupsi}\tau=\zeta\sqrt{{2\over \lambda}\left\{{\zeta-\eta_\infty\over\eta_0-\eta_\infty}\right\}^{2\over r-2}-1},
		\end{equation}
		which has a unique solution denoted by $\zeta=\psi(\tau)$ for $\tau\in\mathbb{R}^+$.

		Moreover, the pressure $\widetilde p\in L^{2}_0(\omega)\cap H^1(\omega)$ is the unique solution to the Darcy problem
		\begin{equation}\label{Reynolds_equation}{\rm div}_{x'}\widetilde V'(x')=0\quad\hbox{in }\omega,\quad \widetilde V'(x')\cdot n=0\quad\hbox{on }\partial\omega.
		\end{equation}

		\item  If $\gamma\neq 1$, the filtration velocity is given by
		\begin{equation}\label{FiltrationVelcovcity0}
			\widetilde V'(x')=-{1\over 6\eta}A(f'(x')-\nabla_{x'}\widetilde p(z')),\quad \widetilde V_3\equiv 0\quad\hbox{in }\omega,
		\end{equation}
		where $\eta$ is equal to $\eta_0$ if $\gamma<1$ or $\eta_\infty$ if $\gamma>1$, and 
		the symmetric and definite positive tensor $A\in\mathbb{R}^{2\times 2}$ is defined by its entries
		\begin{equation}\label{permeabilityA0}
			A_{ij}=\int_{Z'_f}(e^i+\nabla_{z'}\widehat q^i)e_j\,dz',\quad i,j=1,2.
		\end{equation}
		For $i=1,2$,  $\widehat q^i(z')$ denotes the unique solutions in $H^1_\#(Z')$ of the local Hele-Shaw problems in 2D given by  
		\begin{equation}\label{local_problems_0}
			\left\{\begin{array}{rl}
				\displaystyle
				\Delta_{z'}\widehat q^i=0 &\hbox{in }  Z'_f,\\
				\noame
				\displaystyle
				(\nabla_{z'}\widehat q^i+e_i)\cdot n=0&\hbox{on }  \partial T'.
			\end{array}\right.
		\end{equation}

		Moreover, the pressure $\widetilde p\in L^{2}_0(\omega)\cap H^1(\omega)$ is the unique solution of the Darcy problem \begin{equation}\label{Reynolds_equation0}{\rm div}_{x'}\widetilde V'(x')=0\quad\hbox{in }\omega,\quad \widetilde V'(x')\cdot n=0\quad\hbox{on }\partial\omega.
		\end{equation}
	\end{itemize}

\end{theorem}

  \begin{theorem}[Case $r>2$]\label{expresiones_finalesrsup2}
	Depending on the value of $\gamma$, filtration velocity $\widetilde V$ can be expressed as follows:
	\begin{itemize}
		\item If $\gamma<1$, the filtration velocity is given by (\ref{FiltrationVelcovcity0}), with $\eta=\eta_0$ and permeability tensor $A$ given by (\ref{permeabilityA0}) with local problems (\ref{local_problems_0}) and the Darcy problem (\ref{Reynolds_equation0}).
		\item If $\gamma>1$, the filtration velocity is given by
		\begin{equation}\label{FiltrationVelcovcity0r}
			\widetilde V'(x')=-{1\over \lambda^2 (\eta_\infty-\eta_0)^{r'-1}2^{r'\over 2}( r'+1)}\mathcal{U}(f'(x')-\nabla_{x'}\widetilde p(z')),\quad \widetilde V_3\equiv 0\quad\hbox{in }\omega,
		\end{equation}
		where  the permeability function $\mathcal{U}:\mathbb{R}^2\to \mathbb{R}^2$ is monotone and coercive, defined by
		\begin{equation}\label{permeabilityA0r}
			\mathcal{U}(\delta')=\int_{Z'_f}\left|\delta'+\nabla_{z'}\widehat q_{\delta'}\right|^{r'-2}\left(\delta'+\nabla_{z'}\widehat q_{\delta'}\right)\,dz',\quad\forall\delta'\in\mathbb{R}^2.
		\end{equation}
		Here, $\widehat q_{\delta'}(z')$, for every $\delta'\in\mathbb{R}^2$, denotes the unique solution in $ L^{r'}_{0,\#}(Z'_f)\cap W^{1,r'}(Z'_f)$ of the local Hele-Shaw problem 
		\begin{equation}\label{Reynolds_power}
			\left\{\begin{array}{rl}\displaystyle
				{\rm div}_{z'}\left(\left|\delta'+\nabla_{z'}\widehat q_{\delta'}\right|^{r'-2}\left(\delta'+\nabla_{z'}\widehat q_{\delta'}\right)
				\right)=0&\hbox{in }Z_f',\\
				\noame
				\displaystyle \left(\left|\delta'+\nabla_{z'}\widehat q_{\delta'}\right|^{r'-2}\left(\delta'+\nabla_{z'}\widehat q_{\delta'}\right)
				\right)\cdot n=0&\hbox{ on }\partial T'.
			\end{array}\right.
		\end{equation}
		Moreover, the pressure $\widetilde p\in L^{r'}_0(\omega)\cap W^{1,r'}(\omega)$ is the unique solution of the Darcy problem 
		\begin{equation}\label{Reynolds_equation0r2}{\rm div}_{x'}\widetilde V'(x')=0\quad\hbox{in }\omega,\quad \widetilde V'(x')\cdot n=0\quad\hbox{on }\partial\omega.
		\end{equation}
		\item If $\gamma=1$, the filtration velocity is given by (\ref{FiltrationVelcovcity}), with local problems (\ref{expressionq}) and the Darcy problem (\ref{Reynolds_equation}).
		
	\end{itemize}

	\begin{remark}
		Theorems~\ref{expresiones_finales} and~\ref{expresiones_finalesrsup2} can be compared with~\cite[Theorems 2.1 and 2.3]{Anguiano_Bonn_SG_Carreau2} where the case of proportionally thin porous media (PTPM) is considered. It should be noted that the local Stokes problems involving velocity and pressure that appear for the PTPM  are replaced by local Hele-Shaw problems involving only the pressure for the VTPM.
	\end{remark}

\end{theorem}

\subsection{Comments on the proofs and outline of the paper}

The approach developed in this study relies on an adaptation of the unfolding method to the present context, in which three different scales may be identified in the geometric description of the physical domain $\Omega_\eps$:
\begin{itemize}
	\item a macroscale corresponding to the size of the two-dimensional domain $\omega$;
	\item a microscale $\eps$ describing the thickness of the thin fluid layer;
	\item an intermediate or mesoscale $\eps^\ell$ modeling the periodicity of the medium in the horizontal directions.
\end{itemize}

As previously stated, the goal is to determine the limit of the sequence of solutions $(\widetilde u_\eps, \widetilde p_\eps)$ for the rescaled Stokes system~\eqref{system_1_dil}. Because these functions are initially defined over dilated sets $\widetilde \Omega_\eps$ that vary with $\eps$, the first step in the proof is to extend this function to the common limit domain $\Omega$, in such a way that passing to the limit in the variational formulation obtained after applying the unfolding method will be possible. This implies, in particular, the preservation of the a priori estimates satisfied by $(\widetilde u_\eps, \widetilde p_\eps)$.

Because homogeneous Dirichlet boundary conditions are imposed on the velocity, the natural extension by zero preserves the $L^q$ type estimates on $\widetilde u_\eps$ and its derivatives (see Remark~\ref{Rem:extensionVelocity}). However, the treatment of the pressure is more delicate, and constitutes a major novelty of this study. Indeed, following the original idea of Tartar~\cite{Tartar}, a classical way of extending the pressure in problems involving incompressible flows in periodic porous media is to argue by duality. Using this method, the extended pressure is obtained by applying De Rham's theorem to a bounded operator that combines the variational formulation of the problem and a restriction operator from $W^{1,q}(\Omega)^3$ to $W^{1,q}(\widetilde \Omega_\eps)^3$ satisfying adequate estimates.

In the case of the VTPM developed here, this well-known strategy does not seem to be conclusive in the sense that using a restriction operator adapted to this particular geometry, we were not able to derive the optimal estimates of the extended pressure required to pass to the limit in the variational formulation. For this reason, we developed a new method, based on a recent decomposition result for $L^p$ functions defined over a thin domain in~\cite{CLS}, and on an extension theorem for Sobolev spaces in the context of periodic porous domains~\cite{Acerbi}.

The remainder of this paper is organized as follows. Because the geometry of $\Omega_\eps$ is rather unusual, we studied in detail the validity of the hypotheses that are required to apply~\cite[Theorem 3.2 and Corollary 3.4]{CLS} and gathered all the relevant properties of domains $\omega_\eps$ in Section~\ref{sec:geometry}. Section~\ref{Sect:conv} details all results regarding the convergence of the rescaled and extended functions $\widetilde u_\eps,\widetilde p_\eps$ and their unfolded counterparts $\widehat u_\eps, \widehat p_\eps$. Section~\ref{Sec:obtentionLimitModels} concludes the proof of Theorems~\ref{expresiones_finales} and~\ref{expresiones_finalesrsup2}, after establishing the two-pressure limit systems satisfied by the limits $\widehat u,\widetilde p$ of the (properly normalized) unfolded functions $\widehat u_\eps, \widehat p_\eps$ (see Theorems~\ref{thm_pseudoplastic} and~\ref{thm_dilatant}).

\section{Geometric properties satisfied by the domains $\omega_\eps$}\label{sec:geometry}

In this section, we gather some geometric properties satisfied by the sequence of domains $\omega_\eps$, as well as uniform functional inequalities that will play a key role in proving the a priori estimates detailed in Section~\ref{Sect:conv}. More specifically, these properties will allow us to apply to the pressure $p_\eps$ a decompositon result proven in~\cite{CLS}, to prove Lemma~\ref{lemma_est_P}.

We first recall the classical definition of a Lipschitz domain in $\R^2$ (see, for example,~\cite[Definition 2.4.5]{HenrotPierre}).

\begin{definition}\label{Def:Lipschitzdomain}
	Let $U\subset\R^2$ be an open set with compact boundary. We say that $U$ is a \emph{Lipschitz domain} if there exist positive constants $a,r,L$ such that for any $y'\in \partial U$, there exists an orthonormal local coordinate system with origin at $y'=0$ and a function $\eta_{y'}:(-r,r)\to (-a,a)$, of class $C^2$, such that $\eta_{y'}(0)=0$, and setting $K=(-r,r)\times (-a,a)$, $\partial U \cap K$ and $U \cap K$ are respectively described in the new coordinate system by
	\begin{align}
		\partial U \cap K & = \left\lbrace  (X_1,\eta_{y'}(X_1)),\ -a<X_1<a \right\rbrace, \label{RepBdy}\\
		U \cap K & = \left\lbrace  (X_1,X_2)\in K,\ X_2 > \eta_{y'}(X_1) \right\rbrace. \label{RepSet}
	\end{align}
\end{definition}

Next, we clarify the notion of uniform $C^2$ regularity, a property that we assume to be satisfied by the sets $\omega$ and $T'$, and that will be used in the proof of Lemma~\ref{Property3.5}.
\begin{definition}\label{Def:C2domain}
	Let $U\subset\R^2$ be an open set with compact boundary. We say that $U$ is \emph{uniformly $C^2$} if $U$ is a Lipschitz domain and if for every $y'\in \partial U$, the function $\eta=\eta_{y'}$ introduced in Definition~\ref{Def:Lipschitzdomain} is of class $C^2$, and there exists a constant $C>0$ (independent of $y'$) such that
	\begin{equation}\label{RegAssumption}
		\|\eta'\|_{\infty}+\|\eta''\|_{\infty} \leq C.
	\end{equation}
\end{definition}

Lemmas~\ref{Prop:uniformPoincareWirtinger}, \ref{Prop:uniformConeCondition} and~\ref{Property3.5} allow us to apply the decomposition theorem~\cite[Theorem 3.2]{CLS} and its corollary~\cite[Corollary 3.4]{CLS} to derive the decomposition of the pressure detailed in Lemma~\ref{lemma_est_P}.

The first is the well-known Poincar\'e-Wirtinger inequality in $\omega_{\ep}$, with a uniform constant. This result is proven in~\cite[Theorem 2.14]{Ciora3}.
\begin{lemma}\label{Prop:uniformPoincareWirtinger}
	Let $q\in (1,+\infty)$. There exists a constant $C>0$ such that for all $\eps>0$ and $g_\eps\in W^{1,q}(\omega_\eps)$,
	\begin{equation}\label{Ineq:PoincareWirtinger}
		\left\| g_\eps -\fint_{\omega_\eps}g_\eps(y')dy'  \right\|_{L^q(\omega_\ep)} \leq C\|\nabla g_\eps\|_{L^q(\omega_\ep)^q}.
	\end{equation}
\end{lemma}

The next result indicates a uniform cone condition (at scale $\eps$).
Given $h>0$, $\theta\in (0,\pi/2)$ and a unitary vector $\xi'\in \R^2$, we denote by $\mathcal C(\xi',h,\theta)$ the open cone of $\R^2$ with vertex at the origin, angle $2\theta$, height $h$ and directed by $\xi'$, i.e.,
\begin{equation*}
	\mathcal C(\xi',h,\theta) = \{ y'\in \R^2,\ |y'|\cos\theta < x'\cdot \xi' < h \}.
\end{equation*}

\begin{lemma}\label{Prop:uniformConeCondition}
	There exists $\theta\in (0,\pi/2)$ such that for every $\eps>0$ and every $x'\in \overline{\omega_\eps}$, there exists a unit vector $\xi'_{x'}\in \R^2$ such that the following condition holds:
	\begin{equation}\label{Inclusion:cone}
		\forall y'\in B'(x',\eps)\cap\overline{\omega_\eps}\quad y'+\mathcal C(\xi'_{x'},\eps,\theta)\subset \omega_\eps.
	\end{equation}
\end{lemma}

\begin{proof}
	From~\eqref{Assumption:distanceBoundary} and the assumptions that $\omega$ and $T'$ are uniformly $C^2$ in the sense of Definition~\ref{Def:C2domain}, we observe that the dilated domain $\eps^{-\ell}\omega_\eps$ satisfies the following uniform cone property: there exists $\theta\in (0,\pi/2)$, $h>0$ and $r>0$ such that for all $z_0'\in \eps^{-\ell}\omega_\eps$, there exists a unit vector $\xi'\in \R^2$ such that
	\begin{equation*}
		\forall z'\in B'(z_0',r)\cap \eps^{-\ell}\overline{\omega_\eps}\quad z'+\mathcal C(\xi',\theta,h)\subset \eps^{-\ell}\overline{\omega_\eps}.
	\end{equation*}
	Setting $x'=\eps^\ell z_0'$ and $y'=\eps^\ell z'$, this is equivalent to:
	\begin{equation}\label{ProofConeProp}
	\forall y'\in B'(x_0',\eps^\ell r)\cap \overline{\omega_\eps}\quad y'+\eps^\ell \mathcal C(\xi',\theta,h)\subset \overline{\omega_\eps}.
\end{equation}
	Since $\ell<1$, we can assume that $\eps<\eps^\ell \max(r,h)$. Using that  $\eps^\ell \mathcal C(\xi',\theta,h) =  \mathcal C(\xi',\theta,\eps^\ell h)$, we deduce the inclusions
	\[
	B'(x_0',\eps)\subset B'(x_0',\eps^\ell r)\quad \textrm{and}\quad 
   \mathcal C(\xi',\theta,\eps)\subset \eps^\ell \mathcal C(\xi',\theta,h),
		\]
which, combined with~\eqref{ProofConeProp}, yields the desired property~\eqref{Inclusion:cone}.
\end{proof}

The last property that we need is another uniform Poincar\'e-Wirtinger inequality that occurs locally in $\Omega_\eps$. Following the notation from~\cite{CLS}, we define for every $x'\in \overline{\omega_\eps}$ the set $\widehat B(x',\eps)$ as
\begin{equation}\label{Def:Bhat}
	\widehat B(x',\eps)=\left\lbrace y\in \Omega_\eps,\ y'\in B'(x',\eps)  \right \rbrace.
\end{equation}

\begin{lemma}\label{Property3.5}
	Let $q\in (1,+\infty)$ be fixed. There exists a constant $\Lambda>0$ such that for every $\eps>0$ and every $x'\in \overline{\omega_\eps}$,
	\begin{equation}\label{PoincareWirtingerinBhat}
			\forall p_\eps\in L^q(\widehat B(x',\eps))\quad 	\left\| p_\eps -\fint_{\widehat B(x',\eps)}p_\eps(y)dy  \right\|_{L^q(\widehat B(x',\eps))} \leq C\|\nabla p_\eps\|_{L^q(\widehat B(x',\eps))^q}.
	\end{equation}
\end{lemma}
\begin{proof}
	We fix $x'\in \overline{\omega_\eps}$ and assume that $\mathrm{dist}(x',\partial \omega_\eps)\geq \eps$. In this  case, the set $\widehat B(x',\eps)$ defined by~\eqref{Def:Bhat} is simply:
	\[
	\widehat B(x',\eps)=B'(x',\eps)\times (0,\eps),
	\]
	and it can be transformed into the fixed cylinder 
	\[
	Q:=B'(0,1)\times (0,1)
	\]
	by applying the dilatation
	\begin{equation}\label{Dilatation}
		z'=(y'-x')/\eps, \quad z_3=y_3/\eps.
	\end{equation}

	Now, take $p_\eps\in L^q(\widehat B(x',\eps))$ and define $r_\eps\in L^q(Q)$ by
	\begin{equation}\label{Def:reps}
		r_\eps(z)=p_\eps(x'+\eps z',\eps z_3),\quad z\in Q\, .
	\end{equation}
	Writing $\widehat B=\widehat B(x',\eps)$, there hold the relations
	\begin{align}
		\left \| r_\eps - \fint_{Q} r_\eps \right \|_{L^q(Q)}  & = \eps^{-3/q} \left\| p_\eps - \fint_{\widehat B} p_\eps \right\|_{L^q(\widehat B)}, \label{LqnormQ}\\
		\|\nabla r_\eps\|_{W^{-1,q}(Q)} & = \eps^{-3/q} \|\nabla r_\eps\|_{W^{-1,q}(\widehat B)}. \label{WminusnormQ}
	\end{align}
	Equality~\eqref{LqnormQ} is derived from direct computation and the change of variables~\eqref{Dilatation}:
	\begin{align*}
		\left\| r_\eps - \fint_{Q} r_\eps \right \|_{L^q(Q)}^q & = \int_Q \left( r_\eps(z) -\frac{1}{\textrm{Vol}(Q)} \int_{Q} r_\eps(u)du\right)^q dz\\
		& = \eps^{-3}\int_{\widehat B} \left(p_\eps(y) - \frac{\eps^{-3}}{\textrm{Vol}(Q)} \int_{\widehat B} p_\eps(v)dv\right)^q dy\\
		& = \eps^{-3}\int_{\widehat B} \left(p_\eps(y) - \frac{1}{\textrm{Vol}(\widehat B)} \int_{\widehat B} p_\eps(v)dv\right)^q dy\\
		& = \eps^{-3}\left\| p_\eps - \fint_{\widehat B } p_\eps \right \|_{L^q(\widehat B)}^q.
	\end{align*}
	The second equality~\eqref{WminusnormQ} is achieved by fixing $\varphi\in W^{1,q'}_0(Q)$ (with $\frac{1}{q}+\frac{1}{q'}=1$) and associating it (for instance) with the function $\widehat{\varphi}\in W^{1,q}_0(\widehat B)$ defined by
	\begin{equation}\label{ChangeofFunctionsDualityProduct}
		\widehat \varphi(y) = \varphi((y'-x')/\eps, z_3/\eps),\quad y\in \widehat B.
	\end{equation}
	Once again,
	\begin{align*}
		\int_{Q}|\nabla_z \varphi(z)|^{q'} dz 
		& = \eps^{-3+q'}\int_{\widehat B} |\nabla_y \widehat \varphi(y)|^{q'} dy.
	\end{align*}
	Therefore, $\|\varphi\|_{W^{1,q'}_0(Q)} =\eps^{-\frac{3}{q'}+1} \|\widehat \varphi\|_{W^{1,q'}_0(\widehat B)}$ and one can compare the $W^{-1,q}$ norms of $r_\eps$ and $p_\eps$ by writing
	\begin{align*}
		\|\nabla_z r_\eps\|_{W^{-1,q}(Q)}& = \sup_{\|\varphi\|\leq 1} \int_Q r_\eps(y)\, \mathrm{div}_z\, \varphi(z)\, dz\\
		& = \sup_{\|\widehat \varphi\|\leq \eps^{\frac{3}{q'}-1}} \eps^{-3} \int_{\widehat B} p_\eps(y)\, \eps\, \mathrm{div}_y\, \widehat\varphi(y)\, dy\\
		& = \eps^{\frac{3}{q'}-1}\eps^{-2}  \sup_{\|\widehat \varphi\|\leq 1} \int_{\widehat B} p_\eps(y)\,  \mathrm{div}_y\, \widehat\varphi(y)\, dy\\
		& = \eps^{-3/q}\|\nabla_y p_\eps\|_{W^{-1,q}(\widehat B)}.
	\end{align*}
	
	Because $Q$ is a bounded, connected Lipschitz domain, the classical 
	Ne{\v{c}}as inequality yields the existence of $\Lambda>0$ such that for any $r\in L^q(Q)$,
	\begin{equation}\label{Necas}
		\left\| r - \fint_{Q} r \right \|_{L^q(Q)}^q \leq \Lambda \|\nabla r\|_{W^{-1,q}(Q)}.
	\end{equation}
	Taking an arbitrary $p_\eps\in L^q(\widehat B(x',\eps))$, defining $r_\eps$ by~\eqref{Def:reps}, and using the relations~\eqref{LqnormQ}--\eqref{WminusnormQ}, we deduce that~\eqref{PoincareWirtingerinBhat} is satisfied for any $x'\in \overline{\omega_\eps}$ such that $\mathrm{dist(x',\partial \omega_\eps)}\geq \eps$.
	
	\bigskip

	It remains to address the case in which $\mathrm{dist}(x',\partial \omega_\eps)< \eps$.  Based on assumption~\eqref{Assumption:distanceBoundary}, two incompatible situations can occur: either $\mathrm{dist}(x',\partial \omega)< \eps$ (subcase 1) or there exists $k'\in \mathcal K_{\eps}$ such that  $\mathrm{dist}(x',\overline T'_{k',\varepsilon^\ell})< \eps$ (subcase 2).

	\emph{Subcase 1. } If $\mathrm{dist}(x',\partial \omega)<\eps$, the set $\widehat B(x',\eps)$ is no longer a cylinder because $B'(x',\eps)$ is intersected by $\partial\omega$. However, the previous argument can be adapted. Indeed, using the dilatation $y\mapsto \eps^{-1}y$ , one can turn $\widehat B$ into a bounded, connected Lipschitz domain $\eps^{-1}\widehat B$, which is now dependent on $\eps$ and $x'$, but whose Lipschitz constants remain uniformly bounded. As a result, the Ne{\v{c}}as inequality still holds in $\eps^{-1}\widehat B$, with a uniform constant (see for instance~\cite[Chapter IV]{BoyerFabrie}).

	We consider a projection $\pi(x')$ of $x'$ on $\partial \omega$, i.e.\! a point $\pi(x')\in \partial \omega$ satisfying
	\[
	|x'-\pi(x')| = \min\{ |x'-y'|,\ y'\in \partial \omega \}.
	\]
	Since $\omega$ is uniformly $C^2$, we consider the function $\eta$ associated with $y'=\pi(x')$ in Definition~\ref{Def:C2domain}. Because $a,r$ are fixed, upon choosing $\eps$ small, we may assume that $B'(x',\eps)\subset K$, and in particular,
	\[
	[B'(x',\eps)\cap K] \subset [\omega \cap K].
	\]
	We claim that the rescaled set $\eps^{-1}\widehat B(x',\eps)$ is still a connected Lipschitz domain with bounded constants $a,r,L$.

	The idea of the proof is to describe the part $\eps^{-1}(\partial \omega\cap B'(x',\eps))$ as being arbitrary close to a segment, and control the angle between this segment and the circle $\eps^{-1}\partial B'(x',\eps)$ at points that belong to the intersection $\eps^{-1}(\partial \omega\cap\partial B'(x',\eps))$. To this aim, take $y'\in \partial \omega\cap \partial B'(x',\eps)$. In the new coordinate system associated with representations~\eqref{RepBdy}--\eqref{RepSet}, $y'$ is represented by the vector $X'=(X_1,X_2)$ such that
	\[
	X'=R(y'-\pi(x'))
	\]
	where $R\in SL_2(\R)$. In particular, $|X'|=|y'-\pi(x')|$ so by triangle inequality,
	\[
	|X'|\leq |y'-x'|+|x'-\pi(x')| < 2\eps
	\]
	since $y'\in \partial  B'(x',\eps)$ and $|x'-\pi(x')|=\mathrm{dist}(x',\partial \omega)<\eps$. This implies that
	\[
	|X_1|<C\eps.
	\]

	Then, we can use the regularity assumption~\eqref{RegAssumption} to deduce that
	\[
	|\eta(X_1)-\eta(0)-\eta'(0)X_1|\leq \frac{1}{2}|X_1|^2\, \| \eta'' \|_{\infty} \leq C|X_1|^2.
	\]
	Since $\eta(0)=0$, this implies
	\[
	\eps^{-1} |\eta(X_1)-\eta'(0)X_1|\leq C\eps.
	\]
	Notice that one also has
	\[
	|\eta'(X_1)-\eta'(0)|\leq |X_1|\, \|\eta'\|_{\infty}\leq C\eps.
	\]
	This implies that the direction of the tangent vector $\tau=(1,\eta'(X_1))$ is arbitrarily close to the direction of $\tau_0:=(1,\eta'(0))$. 
	
	Finally, denoting by $X^*$ the vector representing $x'$ in the local system of coordinates, one can observe that $X^*$ is the minimizer of the quantity $|X^*-X|^2$ among all $X\in \partial\omega\cap K$. Writing
	\[
	|X^*-X|^2 = |X_1^*-X_1|^2 + |X_2^*-\eta(X_1)|^2
	\]
	it is easy to see that $X^*$ satisfies the necessary optimality condition
	\[
	X_1^*+ \eta'(0)X_2^* = 0.
	\]
	Hence, the vector $\tau_0$ is orthogonal to $X^*$. Elementary geometrical arguments can then be applied to establish a uniform bound on the angle between the vector $\tau$ (which is arbitrarily close to $\tau_0$) and the exterior normal to the ball $B'(x',\eps)$ at any intersection point. Since the dilatation $y'\mapsto \eps^{-1}y'$ preserves the angles, this proves that $\eps^{-1}(\omega\cap B'(x',\eps))$ remains a connected Lipschitz domain with constants independent on $x'$ and $\eps$, and so does $\eps^{-1}\widehat B(x',\eps)$.

	\medskip
	
	\emph{Subcase 2.} Now, consider the case of a point $x'$ such that $\mathrm{dist}(x',T'_{k',\eps^\ell})<\eps$. For simplicity, it is not restrictive to assume that $k'=(0,0)$. Because $T'$ is uniformly $C^2$, one can modify the reasoning from Subcase 1 to consider that the intersection $B'(x',\eps)\cap \partial \omega$ is now replaced by 
	\[
	B'(x',\eps)\cap (\eps^\ell \partial T').
	\]

	Applying the dilatation $Z'=y'/\eps$ to $y'\in \overline{B'(x',\eps)}$, setting $Z^*=x'/\eps$ and using the fact that $T'$ is uniformly $C^2$, we see that the rescaled set $B'(Z^*,1)\cap (\eps^{\ell-1} \partial T')$ is contained in a graph:
	\[
	B'(Z^*,1)\cap (\eps^{\ell-1} \partial T')\subset\left\lbrace 
	(Z_1,Z_2)\in B'(Z^*,1),\ Z_2 = \eta_{\eps}(Z_1)
	\right)\rbrace
	\]
	where $\eta_\eps:(-\eps^{\ell-1}r,\eps^{\ell-1}r)\to (-\eps^{\ell-1}a,\eps^{\ell-1}a)$ is defined by  $\eta_{\eps}(Z_1)=\eps^{\ell-1}\eta(\eps^{1-\ell}Z_1)$. In particular, $\|\eta_{\eps}'\|_{\infty} = \|\eta_{\eps}'\|_{\infty}$ and $\|\eta_{\eps}''\|_{\infty} = \eps^{1-\ell}\|\eta_{\eps}'\|_{\infty}$ so by assumption~\eqref{RegAssumption},
	\[
	\|\eta_{\eps}'\|_{\infty}\leq C\quad \textrm{and}\quad \|\eta_{\eps}''\|_{\infty}\leq C\eps^{1-\ell}.
	\]
	
	Using that $\eta_{\eps}'(0)=\eta'(0)$, we deduce the following estimates: for any $Z'\in  B'(Z^*,1)\cap(\eps^{\ell-1}\partial T'))$ (which satisfies $|Z_1|\leq C$ by triangle inequality),
	\begin{align*}
		|Z_2-\eta'(0)Z_1| \leq C\eps^{1-\ell},\quad 
		|\eta_{\eps}'(Z_1)-\eta_{\eps}'(0)| \leq C\eps^{1-\ell}.
	\end{align*}
	Since $\lim_{\eps\to 0}\eps^{1-\ell}=0$, the above estimates yield that $Z'\in  B'(Z^*,1)\cap(\eps^{\ell-1}\partial T'))$ is composed of a union of finitely many portions of graphs (which are uniformly close to the line passing through the origin and directed by $\eta'(0))$ and finitely many arcs of circles or radius $1$, with a uniform bound on the number of such pieces, and also a uniform control over the angle between the tangent to the graph and the tangent to the arc of circle at any point $Z'\in  \partial B'(Z^*,1)\cap(\eps^{\ell-1}\partial T'))$. This proves that $\eps^{-1}\widehat B(x',\eps)$ is a connected Lipschitz domain with uniform Lipschitz constants, and concludes the proof of Proposition~\ref{Property3.5}.
	
\end{proof}

We conclude this section by stating and proving the existence of a linear and continuous extension operator from $W^{1,p}(\omega_{\ep})$ to $W^{1,p}(\omega)$, which will be used in Corollary~\ref{Coro:extensionp0} to extend the pressure $p_\ep^0$ introduced in Lemma~\ref{lemma_est_P}, to the whole domain $\omega$. This result is based on the results from~\cite{Acerbi}, applied to the domain $\omega$ perforated with periodic holes at scale $\ep^\ell$.

\begin{lemma}\label{Lemma:extensionforpressure}
For every $p\in [1,+\infty)$ and $\eps>0$, there exists a linear continuous extension operator $E_\eps:W^{1,p}(\omega_\eps)\to W^{1,p}(\omega)$ such that 
\begin{equation}\label{Ineq:ExtensionOp}
\forall r_\eps \in W^{1,p}(\omega_\eps)\quad \|E_\eps(r_\eps)\|_{W^{1,p}(\omega)}\leq C \|r_\eps\|_{W^{1,p}(\omega_\eps)}
\end{equation}
where the constant $C>0$ does not depend on $\eps$.
\end{lemma}
\begin{proof}
	According to~\cite[Theorem 2.1]{Acerbi}, there exist $T_\eps:W^{1,p}(\omega_\eps)\to W^{1,p}_{\mathrm{loc}}(\omega)$ and constants $k_0,k_1$ such that for any $r_\eps\in W^{1,p}(\omega_\eps)$, $T_\eps r_\eps = r_\eps$ almost everywhere in $\omega_\eps$, and
	\begin{equation}
		\int_{\omega(\eps^\ell k_0)} |T_\eps r_\eps|^p + |D(T_\eps r_\eps)|^p dx'\leq k_1 	\int_{\omega_\eps} |r_\eps|^p + |Dr_\eps|^p dx'.  \label{EstiTep1}
	\end{equation}
	In the above integral, $\omega(\eps^{\ell}k_0)$ denotes the open subset of $\omega$ defined by
	\[
	\omega(\eps^{\ell}k_0):=\left\lbrace 
	x\in \omega,\ \mathrm{dist}(x,\partial \omega)>\eps^{\ell}k_0
	\right\rbrace.
	\]
	By assumption~\eqref{Assumption:distanceBoundary}, $
	\omega\setminus \omega(\eps^{\ell}k_0)$ is included in $\omega_\eps$,
	so we also have
	\begin{align}
	\int_{\omega\setminus\omega(\eps^\ell k_0)} |T_\eps r_\eps|^p + |D(T_\eps r_\eps)|^p dx' = \int_{\omega\setminus\omega(\eps^\ell k_0)} |r_\eps|^p + |D r_\eps|^p dx' 
	\leq \int_{\omega} |r_\eps|^p + |D r_\eps|^p dx' . \label{EstiTep2}
	\end{align}
	Summing up estimates~\eqref{EstiTep1} and~\eqref{EstiTep2} yields
	\[
	\|T_\eps r_\eps\|_{W^{1,p}(\omega)}\leq (1+k_1) \|r_\eps\|_{W^{1,p}(\omega_\eps)}.
	\]
	Hence, by setting $E_\eps (r_\eps) = T_\eps r_\eps$, we see that $E_\eps$ is a linear continuous extension operator from $W^{1,p}(\omega_\eps)$ to $W^{1,p}(\omega)$ that satisfies the uniform estimate~\eqref{Ineq:ExtensionOp}.
\end{proof}

\section{Convergence of the velocity and pressure}\label{Sect:conv}

In this section, we gather the a priori estimates satisfied by $(\widetilde u_\eps,\widetilde p_\eps)$ and their extensions to $\Omega$, introduce the corresponding unfolded functions and the corresponding estimates, and finally state and prove the convergence results used in Section~\ref{Sec:obtentionLimitModels} to derive the limit models.

\subsection{A priori estimates} \label{sec:estimates}

To derive a priori estimates, we rely on the following well-known results: Poincar\'e and Korn's inequalities in a thin domain of height $\eps$ (see for instance~\cite[Lemmas 0.2 and 0.3]{Tapiero2}).

\begin{lemma}\label{Poincare_lemma} Let $q\in [1,+\infty)$. There exists a constant $C>0$ such that for every $\eps>0$ and  $\varphi\in W^{1,q}_0(\Omega_\varepsilon)^3$, $1\leq q<+\infty$, 
\begin{equation}\label{Poincare}
\|\varphi\|_{L^q(\Omega_\varepsilon)^3}\leq C\varepsilon \|D \varphi\|_{L^q(\Omega_\varepsilon)^{3\times 3}},
\quad
 \|D\varphi\|_{L^q(\Omega_\varepsilon)^{3\times 3}}\leq C\|\mathbb{D}[\varphi]\|_{L^q(\Omega_\varepsilon)^{3\times 3}}.
\end{equation}
As a result, from the change of variables (\ref{dilatacion}),  every $\widetilde\varphi\in W^{1,q}_0(\widetilde\Omega_\varepsilon)^3$ satisfies the following rescaled estimates:
\begin{equation}\label{Poincare2}
\|\widetilde \varphi\|_{L^q(\widetilde \Omega_\varepsilon)^3}\leq C\varepsilon \|D_{\ep} \widetilde \varphi\|_{L^q(\widetilde \Omega_\varepsilon)^{3\times 3}},\quad
 \|D_\varepsilon\widetilde \varphi\|_{L^q(\widetilde \Omega\varepsilon)^{3\times 3}}\leq C\|\mathbb{D}_\varepsilon[\widetilde \varphi]\|_{L^q(\widetilde \Omega_\varepsilon)^{3\times 3}}.
\end{equation}
\end{lemma}

\begin{lemma}\label{lemma_estimates} Depending on the value of $r$, the solution $\widetilde {  u}_\ep$ for the system~\eqref{system_1} satisfies the following estimates.
\begin{itemize}
\item  If $1<r<+\infty$, $r\neq 2$, then
\begin{equation}\label{estim_sol_dil10}
\displaystyle
\|{  u}_\varepsilon\|_{L^2(  \Omega_\varepsilon)^3}\leq C \ep^{{5\over 2}-\gamma}, \quad\displaystyle
\|D  {     u}_\varepsilon\|_{L^2( \Omega_\varepsilon)^{3\times 3}}\leq C\ep^{{3\over 2}-\gamma},\quad\displaystyle
\|\mathbb{D}  [{     u}_\varepsilon]\|_{L^2(\Omega_\varepsilon)^{3\times 3}}\leq C\ep^{{3\over 2}-\gamma}.
\end{equation}
 \item  If $r>2$, depending on the value of  $\gamma$:
 \begin{itemize}
 \item If $\gamma<1$, then
 \begin{equation}\label{estim_sol_dil1_r_sub0}
\displaystyle
\|{    u}_\varepsilon\|_{L^r(  \Omega_\varepsilon)^3}\leq C\ep^{-{2\over r}(\gamma-1)+{r+1\over r}}, \quad\displaystyle
\|D {    u}_\varepsilon\|_{L^r( \Omega_\varepsilon)^{3\times 3}}\leq C\ep^{-{2\over r}(\gamma-1)+{1\over r}},\quad\displaystyle
\|\mathbb{D}  [{    u}_\varepsilon]\|_{L^r( \Omega_\varepsilon)^{3\times 3}}\leq C \ep^{-{2\over r}(\gamma-1)+{1\over r}}.
\end{equation}
\item If $\gamma>1$ then
 \begin{equation}\label{estim_sol_dil1_r_super0}
\displaystyle
\|{   u}_\varepsilon\|_{L^r(  \Omega_\varepsilon)^3}\leq C\ep^{-{\gamma-1\over r-1}+{r+1\over r}}, \quad\displaystyle
\|D  {     u}_\varepsilon\|_{L^r( \Omega_\varepsilon)^{3\times 3}}\leq C\ep^{-{\gamma-1\over r-1}+{1\over r}},\quad\displaystyle
\|\mathbb{D}  [{     u}_\varepsilon]\|_{L^r( \Omega_\varepsilon)^{3\times 3}}\leq C\ep^{-{\gamma-1\over r-1}+{1\over r}}.
\end{equation}
\item If $\gamma=1$ then
 \begin{equation}\label{estim_sol_dil1_r0}
\displaystyle
\|{    u}_\varepsilon\|_{L^r(  \Omega_\varepsilon)^3}\leq C\varepsilon^{1+{1\over r}}, \quad\displaystyle
\|D  {    u}_\varepsilon\|_{L^r( \Omega_\varepsilon)^{3\times 3}}\leq C\ep^{1\over r},\quad\displaystyle
\|\mathbb{D}  [{    u}_\varepsilon]\|_{L^r( \Omega_\varepsilon)^{3\times 3}}\leq C\ep^{1\over r}.
\end{equation}
\end{itemize}
\end{itemize}

Moreover, by applying the change of variables (\ref{dilatacion}), we obtain the following estimates depending on the value of  $r$:
\begin{itemize}
\item  If $1<r<+\infty$, $r\neq 2$, then
\begin{equation}\label{estim_sol_dil1}
\displaystyle
\|{  \widetilde u}_\varepsilon\|_{L^2(\widetilde \Omega_\varepsilon)^3}\leq C \ep^{2-\gamma}, \quad\displaystyle
\|D_{\varepsilon} {   \widetilde u}_\varepsilon\|_{L^2(\widetilde\Omega_\varepsilon)^{3\times 3}}\leq C\ep^{1-\gamma},\quad\displaystyle
\|\mathbb{D}_{\varepsilon} [{   \widetilde u}_\varepsilon]\|_{L^2(\widetilde\Omega_\varepsilon)^{3\times 3}}\leq C\ep^{1-\gamma}.
\end{equation}
 \item  If $r>2$, depending on the value of  $\gamma$:
 \begin{itemize}
 \item If $\gamma<1$, then
 \begin{equation}\label{estim_sol_dil1_r_sub}
\displaystyle
\|{  \widetilde u}_\varepsilon\|_{L^r(\widetilde \Omega_\varepsilon)^3}\leq C\ep^{-{2\over r}(\gamma-1)+1}, \quad\displaystyle
\|D_{\varepsilon} {   \widetilde u}_\varepsilon\|_{L^r(\widetilde\Omega_\varepsilon)^{3\times 3}}\leq C\ep^{-{2\over r}(\gamma-1)},\quad\displaystyle
\|\mathbb{D}_{\varepsilon} [{   \widetilde u}_\varepsilon]\|_{L^r(\widetilde\Omega_\varepsilon)^{3\times 3}}\leq C \ep^{-{2\over r}(\gamma-1)}.
\end{equation}
\item If $\gamma>1$ then
 \begin{equation}\label{estim_sol_dil1_r_super}
\displaystyle
\|{  \widetilde u}_\varepsilon\|_{L^r(\widetilde \Omega_\varepsilon)^3}\leq C\ep^{-{\gamma-1\over r-1}+1}, \quad\displaystyle
\|D_{\varepsilon} {   \widetilde u}_\varepsilon\|_{L^r(\widetilde\Omega_\varepsilon)^{3\times 3}}\leq C\ep^{-{\gamma-1\over r-1}},\quad\displaystyle
\|\mathbb{D}_{\varepsilon} [{   \widetilde u}_\varepsilon]\|_{L^r(\widetilde\Omega_\varepsilon)^{3\times 3}}\leq C\ep^{-{\gamma-1\over r-1}}.
\end{equation}
\item If $\gamma=1$ then
 \begin{equation}\label{estim_sol_dil1_r}
\displaystyle
\|{  \widetilde u}_\varepsilon\|_{L^r(\widetilde \Omega_\varepsilon)^3}\leq C\varepsilon, \quad\displaystyle
\|D_{\varepsilon} {   \widetilde u}_\varepsilon\|_{L^r(\widetilde\Omega_\varepsilon)^{3\times 3}}\leq C,\quad\displaystyle
\|\mathbb{D}_{\varepsilon} [{   \widetilde u}_\varepsilon]\|_{L^r(\widetilde\Omega_\varepsilon)^{3\times 3}}\leq C.
\end{equation}
\end{itemize}
\end{itemize}
\end{lemma}
\begin{proof}
The proof is similar to~\cite[Lemma 3.2]{Anguiano_Bonn_SG_Carreau2}, considering that  the Poincar\'e and Korn inequalities (\ref{Poincare2}) are the same as in the case $\ell=1$ described in \cite{Anguiano_Bonn_SG_Carreau2}.

\end{proof}

\begin{remark}\label{Rem:extensionVelocity}
We extend the velocity $\widetilde u_\eps$ by zero in $\Omega\setminus\widetilde \Omega_\eps$ (which is compatible with the homogeneous boundary condition on $\partial \Omega\cup \partial T_\eps$), and   denote the extension by the same symbol. Obviously, the estimates given in Lemma \ref{lemma_estimates} remain valid, and the extension $\widetilde u_\eps$ is divergence free as well.
\end{remark}

 Next, we decompose the pressure $p_\varepsilon$ using the results from~\cite{CLS} in two pressures $p_\ep^0$ and $p_\ep^1$ and derive  the corresponding estimates. 
\begin{lemma}\label{lemma_est_P} Consider $q=\max\{2,r\}$ and $q'$ the conjugate exponent of $q$, that is, such that $1/q+1/q'=1$. Then, the pressure $p_\ep\in L^q_0(\Omega_\ep)$ solution to~\eqref{system_1} can be decomposed as
\begin{equation}\label{decomposition_original}p_\ep=  p_\ep^0+  p_\ep^1,
\end{equation}
where $p_\ep^0\in W^{1,q'}(\omega_\ep)$ and $p_\ep^1\in L^{q'}(\Omega_\ep)$ can be estimated as follows:
\begin{equation}\label{estim_P_original2}
\|p_\ep^0\|_{W^{1,q'}(\omega_\ep)}\leq C,\quad \|p_\ep^1\|_{L^{q'}(\Omega_\ep)}\leq C\ep^{{1\over q'}+1}.
\end{equation}
Moreover, by applying the change of variable~\eqref{dilatacion}, the rescaled pressure $\widetilde p_\ep^1$ satisfies
\begin{equation}\label{esti_P}
 \|\widetilde p_\ep^1\|_{L^{q'}(\widetilde\Omega_\ep)^3}\leq C\ep.
\end{equation}
\end{lemma}

\begin{proof} 

The proof is divided into two steps. In the first step, we decompose the pressure $p_\ep$ into the sum of two different pressures, $p_\ep^0$ and $p_\ep^1$ and estimate both pressures with respect to the norm of $\nabla p_\ep$ in $W^{-1,q}(\Omega_\ep)^3$. In the second step, we derive estimates for $\nabla p_\ep$, and consequently, for $p_\ep^0$ and $p_\ep^1$.

{\it Step 1. Decomposition of the pressure.} Lemmas~\ref{Prop:uniformPoincareWirtinger},~\ref{Prop:uniformConeCondition} and~\ref{Property3.5} allow us to apply~\cite[Corollary 3.4]{CLS} and deduce that the pressure $p_\ep\in L^{q'}_0(\Omega_\ep)$ can be decomposed as in (\ref{decomposition_original}) with $p_\ep^0\in W^{1,q'}_0(\omega_\ep)$ and $p_\ep^1\in L^{q'}(\Omega_\ep)$, satisfying the following estimates 
\begin{equation}\label{estim_decomposition}\ep^{{1\over q'}+1}\|p_\ep^0\|_{W^{1,q'}(\omega_\ep)}+\|p_\ep^1\|_{L^{q'}(\Omega_\ep)}\leq \|\nabla p_\ep\|_{W^{-1,q'}(\Omega_\ep)}.
\end{equation}
 In~\cite[Theorem 3.2 and Corollary 3.4]{CLS}, the scaling of the pressure is $p_\ep={1\over \ep}\pi_\ep^0+\pi_\ep^1$ satisfying
$$\ep^{{1\over q'}}\|\pi_\ep^0\|_{W^{1,q'}(\omega_\ep)}+  \|\pi_\ep^1\|_{L^{q'}(\Omega_\ep)}\leq C\|\nabla p_\ep\|_{W^{-1,q'}(\Omega_\ep)^3}.$$
Here, we rescale  $p_\ep^0=\ep^{-1} \pi_\ep^0$ and $p_\ep^1= \pi_\ep^1$ and thus, we obtain equation \eqref{estim_decomposition}.
 \\

{\it Step 2. Estimates of $p_\ep^0$ and $p_\ep^1$.} Let us prove the  estimates for the pressures given in equatin \eqref{estim_P_original2}. To do this, according to equation \eqref{estim_decomposition}, we only have to estimate $\nabla p_\eps$. For $\varphi\in W^{1,q}_0(\Omega_\ep)$, the weak formulation of the Stokes system~\eqref{system_1} is
\begin{equation}\label{Fep2}
\begin{array}{rl}
\displaystyle
\langle\nabla p_\ep,\varphi\rangle_{W^{-1,q'}(\Omega_\ep),W^{1,q}_0(\Omega_\ep)}=&\displaystyle -\varepsilon^\gamma(\eta_0-\eta_\infty)\int_{\Omega_\varepsilon}(1+\lambda|\mathbb{D}[{ u}_\varepsilon]|^2)^{{r\over 2}-1}\mathbb{D}[ { u}_\varepsilon]:\mathbb{D}[ \varphi]\,dx\\
\noame
&\displaystyle-\varepsilon^\gamma\eta_\infty\int_{\Omega_\varepsilon}\mathbb{D}[ { u}_\varepsilon]:\mathbb{D}[ \varphi]\,dx+\int_{\Omega_\varepsilon}{ f}'\cdot \varphi'dx.
\end{array}
\end{equation}

First, assume that $1<r<2$, hence $q=q'=2$. Considering that $(1+\lambda|\mathbb{D} [ { u}_\ep |^2)^{{r\over 2}-1}\leq 1$, and applying the Cauchy-Schwarz inequality, we obtain
$$\begin{array}{rl}
\displaystyle \left|\int_{\Omega_\ep}(1+\lambda|\mathbb{D} [ { u}_\ep]|^2)^{{r\over 2}-1}\mathbb{D} [ { u}_\ep]:\mathbb{D} [    \varphi]\,dx\right| \leq &\displaystyle \int_{\Omega_\ep}|\mathbb{D} [ { u}_\ep]||\mathbb{D} [    \varphi] |dx\\
\noame
\leq &\displaystyle \|\mathbb{D} [  { u}_\ep]\|_{L^2( \Omega_\ep)^{3\times 3}}\|\mathbb{D} [  \varphi]\|_{L^2( \Omega_\ep)^{3\times 3}}.
\end{array}$$
Using last estimate in (\ref{estim_sol_dil10}), we
obtain
\begin{equation}\label{Step2_2} 
\displaystyle \left|\ep^\gamma(\eta_0-\eta_\infty)\int_{\Omega_\ep}(1+\lambda|\mathbb{D} [ { u}_\ep]|^2)^{{r\over 2}-1}\mathbb{D} [ { u}_\ep]:\mathbb{D} [ \varphi]\,dx\right| \leq  C\ep^{3\over 2}\| \varphi\|_{H^1_0(\Omega_\ep)^3}
\end{equation}
and
\begin{equation}\label{Step2_4} 
\displaystyle \left|\ep^\gamma \eta_\infty\int_{ \Omega_\ep} \mathbb{D} [ { u}_\ep]:\mathbb{D} [  \varphi]\,dx\right| \leq  C\ep^{3\over 2}\| \varphi\|_{H^1_0(\Omega_\ep)^3}.
\end{equation}
Because ${ f}'={ f}'(x')$ is in $L^\infty(\omega)^2$  and using the Poincar\'e inequality (\ref{Poincare}), we obtain
\begin{equation}\label{Step2_5} 
\displaystyle \left| \int_{ \Omega_\ep} { f}'\cdot   \varphi'\,dx\right| \leq  C\ep^{1\over 2}\|\varphi\|_{L^2(\Omega_\ep)^3}\leq C\ep^{3\over 2}\| D\varphi\|_{L^2(\Omega_\ep)^{3\times 3}}\leq C\ep^{3\over 2}\| \varphi\|_{H^1_0(\Omega_\ep)^3}.
\end{equation}
Returnin to expression (\ref{Fep2}) with $q'=2$, we deduce from (\ref{Step2_2})--(\ref{Step2_5}) the estimate
$$\|\nabla p_\ep\|_{H^{-1}(\Omega_\ep)^3}\leq C\ep^{3\over 2}$$ which, combined with~\eqref{estim_decomposition}, yields (\ref{estim_P_original2}).\\

 	Now, we derive the estimates for the pressure for $r> 2$. In this case, $q=r$ so $q'=r'$. Recall that $\Omega$ is the fixed domain defined in equation~\eqref{OmegaTilde}. Because $r>2$, $L^r(\Omega)$ is continuously embedded in $L^2(\Omega)$, and since the rescaled function $\widetilde \varphi$ is extended by zero to $\Omega$, there exists a constant $C>0$ such that for every $\eps>0$ and $\varphi\in W^{1,r}_0(\Omega_\ep)$,
 	\begin{align*}
 		\|\mathbb{D}_\ep[ \widetilde \varphi]\|_{L^2(\widetilde \Omega_\ep)^{3\times 3}}& = \|\mathbb{D}_\ep[ \widetilde \varphi]\|_{L^2(\Omega)^{3\times 3}}   \\
 		& \leq C \|\mathbb{D}_\ep[ \widetilde \varphi]\|_{L^r(\Omega)^{3\times 3}} = C \|\mathbb{D}_\ep[ \widetilde \varphi]\|_{L^r(\widetilde \Omega_\ep)^{3\times 3}}.
 	\end{align*}
 	Taking into account that for $s\in \{2,r\}$,
 	\[
 	\|\mathbb{D}[\varphi]\|_{L^s(\Omega_\ep)^{3\times 3}}=\ep^{{1\over s}}\|\mathbb{D}_\ep[ \widetilde \varphi]\|_{L^s(\widetilde \Omega_\ep)^{3\times 3}},
 	\]
 	we deduce that for every $\varphi\in W^{1,r}_0(\Omega_\ep)$,
 	\begin{equation}\label{ScaledContinuousLrL2}
 	\|\mathbb{D}[\varphi]\|_{L^2( \Omega_\ep)^{3\times 3}} \leq C\eps^{\frac{1}{2}-\frac{1}{r}} \|\mathbb{D}[\varphi]\|_{L^r( \Omega_\ep)^{3\times 3}}.
 	\end{equation}
 	Using~\eqref{ScaledContinuousLrL2}, H\"older's inequality, and the inequality $(1+X)^{\alpha}\leq C(1+X^\alpha)$ (which is valid for $X\geq 0$, and $\alpha>0$), we obtain
 	$$\begin{array}{l}
 		\displaystyle
 		\int_{ \Omega_\ep}\left|(1+\lambda|\mathbb{D}[  u_\ep]|^2)^{{r\over 2}-1}\mathbb{D}[  u_\ep]:\mathbb{D}[   \varphi]\right|dx\\
 		\noame
 		\displaystyle
 		\leq C\left(
 		\int_{ \Omega_\ep}|\mathbb{D} [ u_\ep]||\mathbb{D} [  \varphi]|\,dx +\int_{  \Omega_\ep}|\mathbb{D} [ u_\ep]|^{r-1}|\mathbb{D} [ \varphi]|\,dx 
 		\right)\\
 		\noame
 		\displaystyle
 		\leq C\left(
 		\|\mathbb{D}[ u_\ep]\|_{L^2( \Omega_\ep)^{3\times 3}}\|\mathbb{D} [ \varphi]\|_{L^2(\Omega_\ep)^{3\times 3}}+\|\mathbb{D}[ u_\ep]\|^{r-1}_{L^r( \Omega_\ep)^{3\times 3}}\|\mathbb{D}[  \varphi]\|_{L^r(  \Omega_\ep)^{3\times 3}}
 		\right)\\
 		\noame
 		\displaystyle\leq C\left( \eps^{\frac{1}{2}-\frac{1}{r}} 		\|\mathbb{D}[  u_\ep]\|_{L^2( \Omega_\ep)^{3\times 3}}+\|\mathbb{D} [ u_\ep]\|^{r-1}_{L^r(\Omega_\ep)^{3\times 3}}
 		\right)\|\mathbb{D}[   \varphi]\|_{L^r(\Omega_\ep)^{3\times 3}}.
 	\end{array}$$
 	By~\eqref{Poincare} with $q=r$, $\|\mathbb{D}[   \varphi]\|_{L^r(\Omega_\ep)^{3\times 3}}\leq C\|\varphi\|_{W^{1,r}_0(\Omega_\ep)}$. Therefore, in the case $\gamma<1$, by combining the previous estimate with the third estimates in~\eqref{estim_sol_dil1} and~\eqref{estim_sol_dil1_r_sub}, we obtain
 	\begin{align*}
 	& 	\left|  \ep^{\gamma}(\eta_0-\eta_{\infty})	\int_{ \Omega_\ep}\left|(1+\lambda|\mathbb{D}[  u_\ep]|^2)^{{r\over 2}-1}\mathbb{D}[  u_\ep]:\mathbb{D}[   \varphi]\right|dx \right| \\
 	&\leq C\eps^{\gamma}\left( \eps^{2-\frac{1}{r}-\gamma} + \eps^{-\frac{2}{r}(\gamma-1)(r-1)+\frac{r-1}{r}} \right) \|\varphi\|_{W^{1,r}_0(\Omega_\ep)}\\
 		& \leq C\left( \eps^{1+\frac{1}{r'}} + \eps^{-\frac{2}{r}(\gamma-1)(r-1)+\frac{1}{r'}+\gamma} \right) \|\varphi\|_{W^{1,r}_0(\Omega_\ep)}.
 	\end{align*}
 	Since $\gamma<1$,  observe that $-{2\over r}(\gamma-1)>-{\gamma-1\over r-1}$, hence $\eps^{-{2\over r}(\gamma-1)}\leq \eps^{-{\gamma-1\over r-1}}$ and
 	\[
 	\eps^{-\frac{2}{r}(\gamma-1)(r-1)+\frac{1}{r'}+\gamma} \leq \ep^{-(\gamma-1)+\frac{1}{r'}+\gamma} = \ep^{1+\frac{1}{r'}}.
 	\]
 	This yields
 	\begin{equation}
 		\left|  \ep^{\gamma}(\eta_0-\eta_{\infty})	\int_{ \Omega_\ep}\left|(1+\lambda|\mathbb{D}[  u_\ep]|^2)^{{r\over 2}-1}\mathbb{D}[  u_\ep]:\mathbb{D}[   \varphi]\right|dx \right|  \nonumber\\
 	\leq  C\ep^{1+\frac{1}{r'}} \|\varphi\|_{W^{1,r}_0(\Omega_\ep)}. \label{step3_1}
 	\end{equation}
 	The other terms are treated very similarly, by writing
 	\begin{align}
 		\left| \eps^{\gamma} \eta_{\infty}\int_{\Omega_\ep} \mathbb{D}[  u_\ep]:\mathbb{D}[  \varphi]\, dx\right|
 		& \leq C\eps^{\gamma} \|\mathbb{D}[  u_\ep]\|_{L^2( \Omega_\ep)^{3\times 3}} \|\mathbb{D}[  \varphi]\|_{L^2( \Omega_\ep)^{3\times 3}} \nonumber\\
 		& \leq C\ep^{\gamma} \ep^{\frac{3}{2}-\gamma} \ep^{\frac{1}{2}-\frac{1}{r}}\|\mathbb{D}[  \varphi]\|_{L^r( \Omega_\ep)^{3\times 3}}\nonumber \\
 		& \leq \ep^{2-\frac{1}{r}}\|\mathbb{D}[  \varphi]\|_{L^r( \Omega_\ep)^{3\times 3}}\nonumber \\
 		&\leq \ep^{1+\frac{1}{r'}}\|\varphi\|_{W^{1,r}_0(\Omega_\ep)}, \label{step3_2}\\
 		\left| \int_{\Omega_\ep}f'\cdot \varphi' \,dx \right| & \leq C \int_{\Omega_\ep}|\varphi|\nonumber\\
 		&\leq C|\Omega_\ep|^{\frac{1}{r'}}\|\varphi\|_{L^r(\Omega_\ep)}\nonumber\\
 		&\leq C \ep^{\frac{1}{r'}}\eps \|\nabla \varphi\|_{L^r(\Omega_\ep)} \nonumber\\
 		& \leq C\ep^{1+\frac{1}{r'}} \|\varphi\|_{W^{1,r}_0(\Omega_\ep)}. \label{step3_3}
 	\end{align}

 Returning to expression~\eqref{Fep2}, we deduce from \eqref{step3_1}--\eqref{step3_3} the  estimate
\begin{equation}\label{estimDpr}\|\nabla p_\ep\|_{W^{-1,r'}(\Omega_\ep)^3}\leq C\ep^{{1\over r'}+1}.
\end{equation}
If $\gamma\geq 1$, by similar arguments, we also deduce (\ref{estimDpr}).
\\

In all cases, using the estimates for the decomposition of $p_\ep$ given in equation~\eqref{estim_decomposition}, we deduce the estimates (\ref{estim_P_original2}).

\end{proof}

The following result is an immediate corollary of estimate~$(\ref{estim_P_original2})_1$ and uniform estimate~\eqref{Ineq:ExtensionOp} satisfied by the extension operator introduced in Lemma~\ref{Lemma:extensionforpressure}.
\begin{corollary}\label{Coro:extensionp0}
Let $E_\eps$ be the extension operator defined in Lemma~\ref{Lemma:extensionforpressure}. There exists a constant $C>0$ such that for every $\eps>0$,
\[
\|E_\eps (p^0_\eps) \|_{W^{1,q'}(\omega)}\leq C.
\]
\end{corollary}

\begin{remark}
	In the sequel, when there is no risk of confusion, we denote by the same symbol $p_\ep^0$ the pressure in $W^{1,q'}(\omega_\ep)$ and its extension to $W^{1,q'}(\omega)$.
\end{remark}

\subsection{Introduction of the unfolded functions and corresponding estimates}\label{sec:unfolding}
 The change of variables (\ref{dilatacion}) does not capture the microstructure of the domain $\widetilde\Omega_\ep$. To do so,  we use an adaptation of the unfolding method (see \cite{Ciora, Ciora2} for more details on the classical version) introduced to this context in~\cite{Anguiano_SG}, and particularize it to the case with period $\ep^\ell$. 
 
 Given $(\widetilde{\varphi}_{\varepsilon},  \zeta_\varepsilon, \widetilde \psi_\varepsilon) \in L^q(\widetilde \Omega_\ep)^3 \times L^{q'}(\omega_\ep)\times L^{q'}(\widetilde \Omega_\ep)$, $1\leq q<+\infty$ and $1/q+1/q'=1$, we define $(\widehat{\varphi}_{\varepsilon},   \widehat \zeta_\varepsilon,\widehat \psi_\varepsilon)\in L^q(\omega\times Z_f)^3\times L^{q'}(\omega\times Z'_f)\times L^{q'}(\omega\times Z_f)$ by
\begin{equation}\label{phihat}
\begin{array}{l}
\displaystyle\widehat{\varphi}_{\varepsilon}(x^{\prime},z)=\widetilde{\varphi}_{\varepsilon}\left( {\varepsilon^\ell}\kappa\left(\frac{x^{\prime}}{{\varepsilon^\ell}} \right)+{\varepsilon^\ell}z^{\prime},z_3 \right),\quad \hbox{a.e. }(x',z)\in \omega\times Z_f,\\
\noame
\displaystyle 
\widehat{\zeta}_{\varepsilon}(x^{\prime},z')= {\zeta}_{\varepsilon}\left( {\varepsilon^\ell}\kappa\left(\frac{x^{\prime}}{{\varepsilon^\ell}} \right)+{\varepsilon^\ell}z^{\prime}\right),\quad \hbox{ a.e. }(x^{\prime},z')\in \omega\times Z'_f,\\
\noame
\displaystyle 
\widehat{\psi}_{\varepsilon}(x^{\prime},z)=\widetilde{\psi}_{\varepsilon}\left( {\varepsilon^\ell}\kappa\left(\frac{x^{\prime}}{{\varepsilon^\ell}} \right)+{\varepsilon^\ell}z^{\prime},z_3 \right),\quad \hbox{ a.e. }(x^{\prime},z)\in \omega\times Z_f.\\
\end{array}\end{equation}
In these definitions, we have extended all functions $\widetilde \varphi_\varepsilon, \zeta_\varepsilon$ and $\widetilde \psi_\varepsilon$ by zero outside $\omega$. The function $\kappa:\mathbb{R}^2\to \mathbb{Z}^2$ is defined by 
$$\kappa(x')=k'\Longleftrightarrow x'\in Z'_{f_{k'},1},\quad\forall\,k'\in\mathbb{Z}^2.$$

\begin{remark}\label{remarkCV}Let us make the following comments related to the definitions~\eqref{phihat}:
\begin{itemize}
\item The function $\kappa$ is well defined up to a set of zero measure in $\mathbb{R}^2$ (the set $\cup_{k'\in \mathbb{Z}^2}\partial Z'_{f_{k'},1}$). Moreover, for every $\varepsilon>0$, we have
$$\kappa\left({x'\over \varepsilon^\ell}\right)=k'\Longleftrightarrow x'\in Z'_{f_{k'},\varepsilon^\ell}.$$

\item For $k^{\prime}\in \mathcal{T}_{\varepsilon}$, the restriction of $(\widehat{\varphi}_{\varepsilon},   \widehat \psi_\varepsilon)$  to $Z^{\prime}_{k^{\prime},{\varepsilon^\ell}}\times Z_f$  (resp. $\widehat \zeta_\varepsilon$ to $Z^{\prime}_{k^{\prime},{\varepsilon^\ell}}\times Z'_f$) does not depend on $x^{\prime}$, whereas as a function of $z$ it is obtained from $(\widetilde{\varphi}_{\varepsilon},  \widetilde{\psi}_{\varepsilon})$  (resp. $\widetilde \zeta_\ep$) using the change of variables 
\begin{equation}\label{CVunfolding}
\displaystyle z^{\prime}=\frac{x^{\prime}- {\varepsilon^\ell}k^{\prime}}{{\varepsilon^\ell}},\end{equation}
which transforms $Z_{f_{k^{\prime}}, {\varepsilon^\ell}}$ into $Z_f$ (resp. $Z'_{f_{k^{\prime}}, {\varepsilon^\ell}}$ into $Z'_f$).
\end{itemize}
\end{remark}

 The next result can be proven in the same manner as in~\cite[Lemma 4.9]{Anguiano_SG}.
\begin{lemma}\label{estimates_relation}
The following estimates relate  $(\widehat \varphi_\varepsilon, \widehat \zeta_\varepsilon, \widehat \psi_\varepsilon)$ to $(\widetilde \varphi_\varepsilon,   \zeta_\varepsilon, \widetilde \psi_\varepsilon)$.
\begin{itemize}
\item  For every $\widetilde\varphi_\varepsilon \in L^q(\widetilde\Omega_\varepsilon)^3$, $1\leq q<+\infty$,
$$\|\widehat \varphi_\varepsilon\|_{L^q(\omega\times Z_f)^3}\leq  \|\widetilde \varphi_\varepsilon\|_{L^q(\widetilde \Omega_\varepsilon)^3},$$
where $\widehat \varphi_\varepsilon$ is given by (\ref{phihat})$_1$. Similarly, for every $\zeta_\varepsilon\in L^{q'}(\omega_\varepsilon)$ and $\widetilde\psi_\ep \in L^{q'}(\widetilde \Omega_\varepsilon)$ the functions $\widehat \zeta_\varepsilon$ and $\widehat \psi_\varepsilon$, respectively given by (\ref{phihat})$_2$ and (\ref{phihat})$_3$ satisfy
$$\|\widehat \zeta_\varepsilon\|_{L^{q'}(\omega\times Z'_f)} \leq \| \zeta_\varepsilon\|_{L^{q'}(\omega_\varepsilon)},\quad\|\widehat \psi_\varepsilon\|_{L^{q'}(\omega\times Z_f)} \leq \|\widetilde \psi_\varepsilon\|_{L^{q'}(\widetilde \Omega_\varepsilon)}.$$
\item  For every $\widetilde \varphi_\varepsilon\in W^{1,q}(\widetilde\Omega_\varepsilon)^3$, $1\leq q<+\infty$, the function $\widehat \varphi_\varepsilon$ given by (\ref{phihat})$_1$ belongs to $L^q(\omega;W^{1,q}(Z_f)^3)$, and
$$\|D_{z'} \widehat \varphi_\varepsilon\|_{L^q(\omega\times Z_f)^{3\times 2}} \leq \varepsilon^\ell  \|D_{x'}\widetilde \varphi_\varepsilon\|_{L^q(\widetilde\Omega_\varepsilon)^{3\times 2}},\quad \|\partial_{z_3} \widehat \varphi_\varepsilon\|_{L^q(\omega\times Z_f)^{3 }} \leq  \|\partial_{z_3}\widetilde \varphi_\varepsilon\|_{L^q(\widetilde\Omega_\varepsilon)^{3}},$$
$$ \|\mathbb{D}_{z'}[\widehat \varphi_\varepsilon]\|_{L^q(\omega\times Z_f)^{3\times 2}} \leq \varepsilon^\ell  \|\mathbb{D}_{x'}[\widetilde \varphi_\varepsilon]\|_{L^q(\widetilde\Omega_\varepsilon)^{3\times 2}},\quad \|\partial_{z_3}[\widehat \varphi_\varepsilon]\|_{L^q(\omega\times Z_f)^{3 }} \leq \|\partial_{z_3}[\widetilde \varphi_\varepsilon]\|_{L^q(\widetilde\Omega_\varepsilon)^{3}}$$
\item  For every $\zeta_\varepsilon\in W^{1,q'}(\omega_\varepsilon)^3$, $1\leq q'<+\infty$, the function $\widehat \zeta_\varepsilon$ given by (\ref{phihat})$_2$ belongs to $L^q(\omega;W^{1,q}(Z'_f)^3)$, and
$$ \|\nabla_{z'}\widehat \zeta_\varepsilon\|_{L^{q'}(\omega\times Z'_f)^{2}} \leq \varepsilon^\ell  \|\nabla_{x'}\zeta_\varepsilon\|_{L^{q'}(\omega_\varepsilon)^{2}}.$$
\end{itemize} 
\end{lemma}

\begin{definition}[Unfolded velocity and pressure] Let us define the unfolded velocity and pressures $(\widehat u_\varepsilon, \widehat p_\varepsilon^0, \widehat p_\ep^1)$ from $(\widetilde u_\varepsilon, p_\varepsilon^0, \widetilde p_\ep^1)$ depending on the value of $r$:
\begin{itemize}
\item If $1<r< 2$,  from $(\widetilde { u}_\varepsilon,   p_\varepsilon^0, \widetilde p_\ep^1)\in H^1_0(\widetilde \Omega_\varepsilon)^3\times H^1(\omega_\varepsilon)\times L^2(\widetilde \Omega_\varepsilon)$, we define $(\widehat { u}_\varepsilon, \widehat p_\varepsilon^0, \widehat p_\varepsilon^1)$ by (\ref{phihat})$_1$ for velocity, (\ref{phihat})$_2$  for pressure zero, (\ref{phihat})$_3$  for pressure one,  and $q=2$.\\

\item If $r>2$, from $(\widetilde { u}_\varepsilon,   p_\varepsilon^0, \widetilde p_\ep^1)\in W^{1,r}_0(\widetilde \Omega_\varepsilon)^3\times W^{1,r'}(\omega_\varepsilon)\times L^{r'}_0(\widetilde \Omega_\varepsilon)$, we define $(\widehat { u}_\varepsilon, \widehat p_\varepsilon^0, \widehat p_\ep^1)$  by (\ref{phihat})$_1$ for velocity, (\ref{phihat})$_2$ for pressure zero,  (\ref{phihat})$_2$ for pressure one, and $q=r$.
\end{itemize}
\end{definition}
Now, by combining estimates of the velocity  and pressures   with  Lemma~\ref{estimates_relation}, we deduce the following estimates of $(\widehat u_\varepsilon,\widehat p_\ep^0,\widehat p_\ep^1)$:

\begin{lemma}\label{estimates_hat} The unfolded $\widehat { u}_\varepsilon$ satisfies the following estimates depending on the value of $r$:

\begin{itemize}
\item  If $1<r<+\infty$, $r\neq 2$, then
\begin{equation}\label{estim_sol_hat1}
\displaystyle
\|{\widehat u}_\varepsilon\|_{L^2(\omega\times Z_f)^3}\leq C \ep^{2-\gamma}, \quad\displaystyle
\|D_{z'} {   \widehat u}_\varepsilon\|_{L^2(\omega\times Z_f)^{3\times 2}}\leq C\ep^{\ell+1-\gamma},\quad\displaystyle
\|\partial_{z_3} {\widehat  u}_\varepsilon\|_{L^2(\omega\times Z_f)^{3}}\leq C\ep^{2-\gamma}.
\end{equation}
 \item  If $r>2$, depending on the value of $\gamma$:
 \begin{itemize}
 \item If $\gamma<1$, then
 \begin{equation}\label{estim_sol_hat1_r_sub}
\displaystyle
\|{  \widehat u}_\varepsilon\|_{L^r(\omega\times Z_f)^3}\leq C\ep^{1-{2\over r}(\gamma-1)}, \quad\displaystyle
\|D_{z'} {  \widehat u}_\varepsilon\|_{L^r(\omega\times Z_f)^{3\times 2}}\leq \ep^{\ell-{2\over r}(\gamma-1)},\quad\displaystyle
\|\partial_{z_3} {\widehat  u}_\varepsilon\|_{L^2(\omega\times Z_f)^{3}} \leq \ep^{1-{2\over r}(\gamma-1)}.
\end{equation}
\item If $\gamma>1$ then
 \begin{equation}\label{estim_sol_hat1_r_super}
\displaystyle
\|{  \widehat u}_\varepsilon\|_{L^r(\omega\times Z_f)^3}\leq C\ep^{1-{\gamma-1\over r-1}}, \quad\displaystyle
\|D_{z'} {   \widehat u}_\varepsilon\|_{L^r(\omega\times Z_f)^{3\times 2}}\leq C\ep^{\ell-{\gamma-1\over r-1}},\quad\displaystyle
\|\partial_{z_3} {\widehat  u}_\varepsilon\|_{L^2(\omega\times Z_f)^{3}}\leq C\ep^{1-{\gamma-1\over r-1}}.
\end{equation}
\item If $\gamma=1$ then
 \begin{equation}\label{estim_sol_hat1_r}
\displaystyle
\|{  \widehat u}_\varepsilon\|_{L^r(\omega\times Z_f)^3}\leq C\varepsilon, \quad\displaystyle
\|D_{z'} {   \widehat u}_\varepsilon\|_{L^r(\omega\times Z_f)^{3\times 2}}\leq C\ep^\ell,\quad\displaystyle
\|\partial_{z_3} {\widehat  u}_\varepsilon\|_{L^2(\omega\times Z_f)^{3}}\leq C\ep.
\end{equation}
\end{itemize}

\end{itemize}
Moreover, considering $q=\max\{2,r\}$ and $q'$ the conjugate exponent of $q$, that is, such that $1/q+1/q'=1$,  it holds that
\begin{equation}\label{estim_P_hat}
 \|\widehat p_\varepsilon^0\|_{L^{q'}(\omega\times Z'_f)}\leq C,\quad \|\nabla_{z'}\widehat p_\varepsilon^0\|_{L^{q'}(\omega\times Z'_f)}\leq C\ep^\ell,\quad \|\widehat p_\varepsilon^1\|_{L^{q'}(\omega\times Z_f)}\leq C\ep.
 \end{equation}

 \end{lemma}

 \subsection{Compactness results on velocities and pressures}\label{sec:convergences}
In this subsection, we analyze the asymptotic behaviour of   $(\widetilde { u}_\ep, p^0_\ep, \widetilde p^1_\ep)$  and the corresponding unfolded functions $(\widehat { u}_\ep, \widehat p_\ep^0, \widehat p_\ep^1)$, when $\ep$ tends to zero.

\begin{lemma}[Convergences of velocities] \label{lem_conv_vel}The extension of the velocity $\widetilde { u}_\ep$ and the unfolded velocity $\widehat { u}_\ep$ satisfy the following convergence results depending on the value of $r$.
\begin{itemize}
\item In the case $1< r< 2$ for every value of $\gamma$, and in the case $r>2$ for $\gamma<1$ , there exist $\widetilde { u}\in V_{z_3}^2(\Omega)$ where $\widetilde { u}=0$ on $\Gamma_0\cup \Gamma_1$ and $\widetilde u_3\equiv 0$, and $\widehat { u}\in  V_{z_3,\#}^2(\omega\times Z_f)^3$, where $\widehat { u}=0$ on $\omega\times (\widehat \Gamma_0\cup \widehat  \Gamma_1)$, $\widehat { u}=0$ on $\omega\times T$ and  $\widehat u_3\equiv 0$, such that, up to a subsequence,
\begin{equation}\label{conv_vel_tilde2}
\ep^{\gamma-2}\widetilde { u}_\ep\rightharpoonup  \widetilde { u} \quad\hbox{in }V_{z_3}^2(\Omega)^3,
\end{equation}
\begin{equation}\label{conv_vel_hat2}
\ep^{\gamma-2}\widehat { u}_\ep\rightharpoonup \widehat { u}\quad\hbox{in }V_{z_3}^2(\omega\times Z_f)^3.
\end{equation}


\item If $r> 2$, then there exist $\widetilde { u}\in V_{z_3}^r(\Omega)$ where $\widetilde { u}=0$ on $\Gamma_0\cup \Gamma_1$ and $\widetilde u_3\equiv 0$, and $\widehat { u}\in  V_{z_3,\#}^r(\omega\times Z_f)^3$, where $\widehat { u}=0$ on $\omega\times (\widehat \Gamma_0\cup \widehat  \Gamma_1)$ and $\widehat u_3\equiv 0$, such that depending on the value of $\gamma$ and up to a subsequence,
\begin{itemize}
\item if $\gamma> 1$,
\begin{equation}\label{conv_vel_tilde_r_super}
\ep^{{\gamma-r\over r-1}}\widetilde { u}_\ep\rightharpoonup \widetilde { u} \quad\hbox{in }V_{z_3}^r(\Omega)^3,
\end{equation}
\begin{equation}\label{conv_vel_hat_r_super}
\quad 
 \ep^{{\gamma-r\over r-1}}\widehat { u}_\ep\rightharpoonup  \widehat { u}\quad\hbox{in }V_{z_3}^r(\omega\times Z_f)^3,
\end{equation}
 \item if $\gamma=1$,
\begin{equation}\label{conv_vel_tilde_r_eq}
\ep^{-1}\widetilde { u}_\ep\rightharpoonup \widetilde { u} \quad\hbox{in }V_{z_3}^r(\Omega)^3,
\end{equation}
\begin{equation}\label{conv_vel_hat_r_eq}
\ep^{-1}\widehat { u}_\ep\rightharpoonup  \widehat { u}\quad\hbox{in }V_{z_3}^r(\omega\times Z_f)^3.
\end{equation}
\end{itemize}
\end{itemize}
Moreover, in every case, the following incompressibility conditions hold: 
\begin{equation}\label{div_conv}
{\rm div}_{x'}\left(\int_0^{1}\widetilde { u}'(x',z_3)\,dz_3\right)=0\quad \hbox{in }\omega,\quad\left(\int_0^{1}\widetilde { u}'(x',z_3)\,dz_3\right)\cdot n=0\quad \hbox{on }\partial\omega,
\end{equation}
\begin{equation}\label{div_conv_z}
{\rm div}_{z'}(\widehat { u}')=0\quad\hbox{in }\omega\times Z_f.
\end{equation}
The limits $\widetilde u$ and $\widehat u$ are related by
\begin{equation}\label{relation}
\widetilde { u}(x',z_3)=\int_{Z'_f}\widehat { u}(x',z)\,dz',\quad \textrm{a.e. }(x',z_3)\in \Omega,
\end{equation}
and as a result, 
\begin{equation}\label{relation2}
\int_0^{1}\widetilde { u}(x',z_3)\,dz_3=\int_{Z_f}\widehat { u}(x',z)\,dz,\quad \textrm{a.e. }x'\in \omega.
\end{equation}
\end{lemma}
 \begin{proof}

  Let us prove the convergence \eqref{conv_vel_tilde2} for $\widetilde u_\ep$. We begin with case $1<r<2$. In this case, from (\ref{estim_sol_dil1}), we obtain the following estimates for the extension of velocity $\widetilde u_\ep$:
\begin{equation}\label{estim_proof2}\|\widetilde{  u}_\ep\|_{L^2(\Omega)^3}\leq C\varepsilon^{2-\gamma},\quad \|\partial_{z_3}\widetilde{  u}_\ep\|_{L^2(\Omega)^3}\leq C\varepsilon^{2-\gamma}.
\end{equation}
Hence, there exists $\widetilde { u}\in V_{z_3}^2(\Omega)^3$ such that 
\begin{equation}\label{conv_vel_proof_1}
\ep^{\gamma-2}\widetilde { u}_\ep\rightharpoonup \widetilde{ u}\quad\hbox{in }V_{z_3}^2(\Omega)^3.
\end{equation}
Also, the continuity of the trace application from $V_{z_3}^Z(\Omega)$ to $L^2(\Gamma_0\cup \Gamma_1)$ implies $\widetilde{ u}=0$ on $\Gamma_0\cup \Gamma_1$. 

Next, taking $\widetilde \varphi\in C^1_c(\Omega)$, multiplying the divergence condition ${\rm div}_\ep(\widetilde{ u}_\ep)=0$ in $\Omega$ by $\ep^{\gamma-1}\widetilde \varphi$ and integrating by parts with respect to $x'$ yields
\[
\int_{\Omega}\partial_{z_3}(\ep^{\gamma-2}\widetilde u_{\eps,3})\widetilde{\varphi} = \int_{\Omega}\ep^{\gamma-1}\widetilde u_{\ep}'\cdot \nabla_{x'}\widetilde{\varphi}.
\]
From the first estimate in~\eqref{estim_proof2}, $\ep^{\gamma-1}\widetilde u_{\ep}'$ converges to zero in $L^2(\Omega)^2$ so we can pass to the limit in the previous equality and obtain using~\eqref{conv_vel_proof_1} the relation
\[
\int_{\Omega}\partial_{z_3}\widetilde u_3\, \widetilde{\varphi} = 0.
\] 
Therefore, $\widetilde u_3$ is independent of $z_3$. Because $\widetilde u_3=0$ on $\Gamma_0\cup \Gamma_1$, we obtain that $\widetilde u_3\equiv 0$, which concludes the proof of convergence~\eqref{conv_vel_tilde2}.

Now, take a test function $\widetilde \varphi(x') \in C^1(\overline \omega)$, which is independent of $z_3$. Multiplying the divergence condition ${\rm div}_\ep(\widetilde{ u}_\ep)=0$ by $\ep^{\gamma-2}\widetilde \varphi$ and integrating by parts, we obtain
\[
\int_{\Omega}\ep^{\gamma-2}\widetilde { u}'_\ep\cdot \nabla_{x'}\widetilde \varphi\,dx'dz_3=0.
\]
Passing to the limit and using Fubini's theorem, we deduce
\[
\int_{\omega}\left( \int_0^1\widetilde u'(x',z_3)\,dz_3\right)\cdot \nabla_{x'}\widetilde \varphi(x')\, dx' = 0,
\]
which proves~\eqref{div_conv} after integrating by parts in the $x'$ variable.
\\

In the case $r>2$ and $\gamma<1$, the velocity $\widetilde u_\ep$ satisfies two types of estimates: estimate (\ref{estim_sol_dil1}) in $L^2$ and (\ref{estim_sol_dil1_r_sub}) in $L^r$. Considering that $2-\gamma>-{2\over r}(\gamma-1)$, estimates (\ref{estim_sol_dil1}) are in fact the optimal ones, that is, estimates (\ref{estim_proof2}). Therefore, we proceed as we did in case $1<r<2$ to complete the proof of convergence~\eqref{conv_vel_tilde2} and the incompressibility condition (\ref{div_conv}).

Now, we prove (\ref{conv_vel_hat2}). In cases $1<r<2$, or  $r>2$ and $\gamma<1$, we recall that by estimate~\eqref{estim_sol_hat1},
\begin{equation}\label{estim_sol_hat1_proof}
	\displaystyle
	\|{  \widehat u}_\varepsilon\|_{L^2(\omega\times Z_f)^3}\leq C \ep^{2-\gamma}, \quad\displaystyle
	\|\partial_{z_3} {\widehat  u}_\varepsilon\|_{L^2(\omega\times Z_f)^{3}}\leq C\ep^{2-\gamma}.
\end{equation}
Since estimates~\eqref{estim_sol_hat1_proof} are similar to~\eqref{estim_proof2}, we deduce by an analog reasoning that the convergence~\eqref{conv_vel_hat2} holds, and also that 
$\widehat u_3\equiv 0$,  $\widehat u'=0$ on $\omega\times T$, and $\widehat u'=0$ on $\omega\times (\widehat \Gamma_0\cup\widehat \Gamma_1)$.  
\\

Let us prove the divergence condition (\ref{div_conv_z}) for case $1<r<2$ and case $r>2$ with $\gamma<1$. Applying the change of variables (\ref{CVunfolding}) to ${\rm div}_{\ep}(\widetilde u_\ep)=0$, we have
$$\ep^{-\ell}{\rm div}_{z'}(\widehat u_\ep')+\ep^{-1}\partial_{z_3}\widehat u_{\ep,3}=0\quad\hbox{in }\omega\times Z_f.$$
Multiplying the previous equality by the test function $\ep^{\ell+\gamma-2}\widehat \varphi \in C^1_c(\omega\times Z'_f)$ and integrating by parts, we obtain
$$\int_{\omega\times Z_f}\ep^{\gamma-2}\widehat { u}'_\ep\cdot \nabla_{z'}\widehat \varphi\,dx'dz=0.$$
Passing to the limit as $\ep$ tends to zero, we deduce (\ref{div_conv_z}).

The proof of the periodicity of $\widehat u$ is similar to that given in \cite{SG_nonlinear}; therefore, we omit it.\\

Finally, we consider the case $r>2$ and $\gamma\geq 1$. In this case, we recall the following estimates for the extension of the velocity $\widetilde u_\ep$:
$$
\displaystyle
\|{  \widetilde u}_\varepsilon\|_{L^r(  \Omega)^3}\leq C\ep^{1-{\gamma-1\over r-1}}, \quad\displaystyle
\|\partial_{z_3} {   \widetilde u}_\varepsilon\|_{L^r(\Omega)^{3\times 3}}\leq C\ep^{1-{\gamma-1\over r-1}},
$$
and the following ones for $\widehat u_\ep$:
$$
\displaystyle
\|{  \widehat u}_\varepsilon\|_{L^r(\omega\times Z_f)^3}\leq C\ep^{1-{\gamma-1\over r-1}}, \quad\displaystyle
\|\partial_{z_3} {\widehat  u}_\varepsilon\|_{L^2(\omega\times Z_f)^{3}}\leq C\ep^{1-{\gamma-1\over r-1}}.
$$

 \end{proof}

\begin{lemma}[Convergences of pressures]\label{lem_conv_press} Consider $q=\max\{2,r\}$ and $q'$ the conjugate exponent of $q$, that is, such that $1/q+1/q'=1$. Then, there exist three functions $\widetilde p\in L^{q'}_0(\omega)\cap W^{1,q'}(\omega)$, independent of $z$, $\widehat p_0\in L^{q'}(\omega;W^{1,q'}_{\#}(Z'_f))$ and $\widehat p_1\in L^{q'}(\omega;L^{q'}_{\#}(Z_f))$, such that
\begin{equation}\label{pressuresq}
  \widehat p_\ep^0\to \widetilde p\quad\hbox{in }L^{q'}(\omega;W^{1,q'}(Z'_f)),\quad \ep^{-\ell}\nabla_{z'}\widehat p_\ep^0\rightharpoonup \nabla_{x'}\widetilde p+\nabla_{z'}\widehat p_0\quad\hbox{in }L^{q'}(\omega;L^{q'}( Z'_f)),
\end{equation}
\begin{equation}\label{pressures_hatq}
\ep^{-1}\widehat p_\ep^1\rightharpoonup \widehat p_1\quad\hbox{in }L^{q'}(\omega;L^{q'}(Z_f)).
\end{equation}
\end{lemma}

\begin{proof}  From the estimates for $p_\ep^0$ and $\widehat p_\ep^0$, and the classical compactness result for the unfolding method for a bounded sequence in $W^{1,q'}$ (see for instance \cite[Proposition 4.57]{Cioran-book}), we obtain convergences~\eqref{pressuresq}. 

The third estimate of (\ref{estim_P_hat}) for $\widehat p_\ep^1$ implies the existence of $\widehat p_1\in L^{q'}_\#(\omega\times Z)$ such that up to a subsequence, the last convergence in (\ref{pressures_hatq}) holds.

 Because $\widetilde p_\ep$ has mean value zero in $\widetilde\Omega_\ep$, from the decomposition of the pressure and the unfolding change of variables, we have
$$0=\int_{\omega\times Z'_f}  \widehat p_\ep^0\,dx'dz'+\int_{\omega\times Z_f}\widehat p_\ep^1\,dx'dz.$$
Considering the convergence of $\widehat p_\ep^0$ to $\widetilde p$, that $\widehat p_\ep^1$ tends to zero,
and $\widetilde p$ does not depend on $z'$, we obtain
$$\int_{\omega}\widetilde p\,dx'=0,$$
and so, $\widetilde p$ has mean value zero in $\omega$.

\end{proof}

\section{Obtention of the limit systems}\label{Sec:obtentionLimitModels}
 
 To complete the proof of Theorems~\ref{expresiones_finales} and~\ref{expresiones_finalesrsup2}, we first use a monotonicity argument that allows us to derive a variational inequality satisfied by the unfolded functions $\widehat u_\eps, \widehat p_\eps$. Then, using the appropriate scaling for the test functions (depending on $r$ and $\gamma$), we derive the two-pressure problem satisfied by the limits $\widehat u, \widetilde p$. These results are contained in Theorems~\ref{thm_pseudoplastic} and~\ref{thm_dilatant}, from which we finally deduce Theorems~\ref{expresiones_finales} and~\ref{expresiones_finalesrsup2}.
 
\subsection{Variational inequality satisfied by the unfolded functions}\label{sec:limitmodel}

Using a monotonicity argument together with Minty's lemma (see, for instance, \cite{Tapiero2, EkelandTemam}), we first derive a variational inequality that will be useful in the proof of the main theorems.  
 
According to Lemma \ref{lem_conv_vel}, we choose a test function ${ v}(x',z)\in \mathcal{D}(\omega; C^\infty_{\#}(Z)^3)$ with $v_3\equiv 0$,     ${ v}'(x',z)=0$ in $\omega\times T$ and    ${ v}'(x',z)=0$ on $\omega\times (\widehat \Gamma_0\cup \widehat \Gamma_1)$, and ${\rm div}_{x'}(\int_Z { v}')\,dz=0$ in $\omega$, $(\int_{Z}{ v}'\,dz )\cdot n=0$ on $\partial\omega$ and ${\rm div}_{z'}({ v}')=0$ in $\omega\times Z$. Multiplying the rescaled Stokes system (\ref{system_1_dil}) by $\widetilde v_\eps(x',z_3)={ v}(x',x'/\varepsilon^\ell,z_3)$ (which belongs to $W^{1,r}_0(\widetilde \Omega_\eps)^3$, $1<r<+\infty$), integrating by parts,  we have
\begin{equation}\label{fomrarvaran}\begin{array}{l}
\displaystyle
\medskip
\varepsilon^\gamma (\eta_0-\eta_\infty)\int_{ \widetilde \Omega_\ep }(1+\lambda|\mathbb{D}_\varepsilon[\widetilde { u}_\varepsilon]|^2)^{{r\over 2}-1}\mathbb{D}_\varepsilon[\widetilde { u}_\varepsilon] :\left( \mathbb{D}_{x'}[ { v}]+\ep^{-\ell}\mathbb{D}_{z'}[{ v}] + \varepsilon^{-1}\partial_{z_3}[ { v}]\right)\,dx'dz_3\\
\medskip
\displaystyle
+\varepsilon^\gamma \eta_\infty\int_{ \widetilde \Omega_\ep}\mathbb{D}_\varepsilon[\widetilde { u}_\varepsilon] :\left( \mathbb{D}_{x'}[ { v}']+\ep^{-\ell}\mathbb{D}_{z'}[{ v}] + \varepsilon^{-1}\partial_{z_3}[ { v}]\right)\,dx'dz_3\\
\medskip
\displaystyle
-\int_{\widetilde\Omega_\ep} p_\varepsilon^0\,{\rm div}_{x'}(v')\,dx'dz_3-\ep^{-\ell}\int_{\widetilde\Omega_\ep}p_\ep^0 {\rm div}_{z'}(v')\,dx'dz_3
\\
\noame
\displaystyle-\int_{\widetilde \Omega_\ep}\widetilde p_\varepsilon^1 \left({\rm div}_{x'}({ v}')+\ep^{-\ell} {\rm div}_{z'}({ v}')+\ep^{-1}\partial_{z_3}v_3\right)\,dx'dz_3=\int_{\widetilde \Omega_\ep} { f}'\cdot { v}'\,dx'dz_3.
\end{array}
\end{equation}
By the change of variables given in Remark \ref{remarkCV}, and taking into account that ${\rm div}_{z'}(v')=0$, we obtain
\begin{equation}\label{form_var_hat}\begin{array}{l}
\displaystyle
\medskip
 \ep^\gamma(\eta_0-\eta_\infty)\int_{ \omega\times Z_f }(1+\lambda |\varepsilon^{-\ell} \mathbb{D}_{z'}[\widehat { u}_\varepsilon] +\ep^{-1}\partial_{z_3}[\widehat { u}_\ep]|^2)^{{r\over 2}-1}\left(\varepsilon^{-\ell} \mathbb{D}_{z'}[\widehat { u}_\varepsilon] +\ep^{-1}\partial_{z_3}[\widehat { u}_\ep]\right):( \ep^{-\ell}\mathbb{D}_{z'}[{ v}] + \ep^{-1}\partial_{z_3}[ { v}]])\,dx'dz\\
\medskip
\displaystyle
+ \ep^\gamma \eta_\infty\int_{ \omega\times Z_f}\left(\varepsilon^{-\ell} \mathbb{D}_{z'}[\widehat { u}_\varepsilon] +\ep^{-1}\partial_{z_3}[\widehat { u}_\ep]\right):( \ep^{-\ell}\mathbb{D}_{z'}[{ v}] + \ep^{-1}\partial_{z_3}[ { v}])\,dx'dz\\
\medskip
\displaystyle
-\int_{\omega\times Z_f} \widehat p_\varepsilon^0\,{\rm div}_{x'}(v')\,dx'dz -\int_{\omega\times Z_f}\widehat p_\varepsilon^1\,(\varepsilon^{-\ell}{\rm div}_{z'}({ v}')+\ep^{-1}\partial_{z_3}v_3)\,dx'dz=\int_{\omega\times Z_f} { f}'\cdot { v}'\,dx'dz+O_\varepsilon,
\end{array}
\end{equation}
where $O_\varepsilon$ is a generic real sequence depending on $\varepsilon$ satisfying $|O_\eps|\leq C\eps^\ell$, that can change from line to line, and devoted to tend to zero. 

\begin{remark}
We point out that in the above variational formulations,  the test function has the form $v(x',x'/\varepsilon^\ell,z_3)$ in the integrals in $\widetilde \Omega_\eps$, and after the change of variables,  it has the form $v(x',z)$  in the integrals in $\omega\times Z_f$. The remaining terms are in $O_\eps$, considering that for $v(x',z)$ in $L^r(\Omega)$ with $1<r<+\infty$, then $\widehat v_\eps \to v$ in $L^r(\omega\times Z_f)$, because  $\|\widehat v_\eps-v\|_{L^r(\omega\times Z_f)}\leq C\eps^\ell$ (see \cite[Proposition 4.4]{Cioran-book} for more details).	
\end{remark}

Now, let us define the functional $J_q$ by
$$
J_q({ v})={2(\eta_0-\eta_\infty)\over r\lambda}\int_{\omega\times Z_f}(1+\lambda|\varepsilon^{1-\ell} \mathbb{D}_{z'}[  { v}] + \partial_{z_3}[ { v}]|^2)^{r\over 2}dx'dz+{\eta_\infty\over 2}\int_{\omega\times Z_f}|\varepsilon^{1-\ell} \mathbb{D}_{z'}[  { v}] + \partial_{z_3}[ { v}]|^2dx'dz.
$$
Observe that for every $\ep>0$, $J_r$ is convex and Gateaux differentiable on $L^q(\omega;W^{1,q}_{\#}(Z_f)^3)$ with  $q=\max\{2,r\}$ (which is an adaptation of \cite[Proposition 2.1 and Section 3]{Baranger}), and $A_r=J'_r$ is given by
\begin{equation}\label{defAr}\begin{array}{l}\displaystyle
(A_r({ w}),{ v})\\
\noame
=\displaystyle (\eta_0-\eta_\infty)\int_{\omega\times Z_f}(1+\lambda|\varepsilon^{1-\ell} \mathbb{D}_{z'}[  { w}] + \partial_{z_3}[ { w}]|^2)^{{r\over 2}-1}(\varepsilon^{1-\ell} \mathbb{D}_{z'}[  { w}] + \partial_{z_3}[ { w}]):(\varepsilon^{1-\ell} \mathbb{D}_{z'}[  { v}] + \partial_{z_3}[ { v}])dx'dz\\
\noame
\displaystyle+\eta_\infty\int_{\omega\times Z_f}(\varepsilon^{1-\ell} \mathbb{D}_{z'}[  { w}] + \partial_{z_3}[ { w}]):(\varepsilon^{1-\ell} \mathbb{D}_{z'}[  { v}] + \partial_{z_3}[ { v}])dx'dz.
\end{array}
\end{equation}
Applying \cite[Proposition 1.1., p.158]{Lions2}, $A_r$ is monotone, that is,
\begin{equation}\label{monotonicity}
(A_r({ w})-A_r({ v}),{ w}-{ v})\ge0,\quad \forall { w},{ v}\in L^q(\omega;W^{1,q}_{\#}(Z_f)^3).
\end{equation}

\subsection{Obtention of the limit models in case $1<r<2$ (pseudoplastic fluids)}

To prove our main results (Theorems~\ref{expresiones_finales} and~\ref{expresiones_finalesrsup2}), we first identify the two-pressure problem satisfied by the limit $\widehat u, \widetilde p$ of the appropriately normalized sequence of rescaled functions $\widehat u_\eps, \widehat p_\eps$ (see Lemma~\ref{lem_conv_vel}).

\begin{theorem}[Case $1<r<2$]\label{thm_pseudoplastic} The pair $(\widehat u, \widetilde p)\in V_{z_3,\#}^2(\omega\times Z_f)^2\times (L^{2}_0(\omega)\cap H^1(\omega))$ given in Lemmas \ref{lem_conv_vel} and \ref{lem_conv_press} is the unique solution of a  reduced two-pressure problem depending on the value of $\gamma$:
\begin{itemize}
\item If $\gamma\neq 1$, then $(\widehat u, \widetilde p)$ satisfies the following reduced Stokes problem
\begin{equation}\label{limit_model_gamma_neq_1}
\left\{\begin{array}{rl}
\displaystyle
-{\eta\over 2}\partial_{z_3}^2\widehat u+\nabla_{z'}\widehat \pi(z')={ f}'(x')-\nabla_{x'}\widetilde p(x') &\hbox{in }\omega\times Z_f,\\
\noame
\displaystyle
{\rm div}_{z'}(\widehat { u}')=0&\hbox{in }\omega\times Z_f,\\
\noame
\displaystyle
{\rm div}_{x'}\left(\int_{Z_f}\widehat u'\,dz\right)=0&\hbox{in }\omega,\\
\noame
\displaystyle \left(\int_{Z_f}\widehat u'\,dz\right)\cdot n=0&\hbox{on }\partial\omega,\\
\noame
\displaystyle \widehat u'=0&\hbox{on } \omega\times (\widehat\Gamma_0\cup\widehat\Gamma_1),\\
\noame
\displaystyle \widehat u'=0&\hbox{on }  \omega\times T,\\
\noame
\displaystyle \widehat \pi\in L^{2}(\omega;L^{2}_{0,\#}(Z'_f)),
\end{array}\right.
\end{equation}
where $\eta$ equal to $\eta_0$ if $\gamma<1$ or $\eta_\infty$ if $\gamma>1$.
\item If $\gamma = 1$, then $(\widehat u, \widetilde p)$ satisfies the following reduced Carreau Stokes problem
\begin{equation}\label{limit_model_gamma_eq_1}
\left\{\begin{array}{rl}
\displaystyle
-{1\over 2}\partial_{z_3} \left((\eta_0-\eta_\infty)\left(1+{\lambda\over 2}|\partial_{z_3} \widehat { u}'|^2\right)^{{r\over 2}-1}+\eta_\infty\right)\partial_{z_3}\widehat { u}'  +\nabla_{z'}\widehat \pi(z')={ f}'(x')-\nabla_{x'}\widetilde p(x') &\hbox{in }\omega\times Z_f,\\
\noame
\displaystyle
{\rm div}_{z'}(\widehat { u}')=0&\hbox{in }\omega\times Z_f,\\
\noame
\displaystyle
{\rm div}_{x'}\left(\int_{Z_f}\widehat u'\,dz\right)=0&\hbox{in }\omega,\\
\noame
\displaystyle \left(\int_{Z_f}\widehat u'\,dz\right)\cdot n=0&\hbox{on }\partial\omega,\\
\noame
\displaystyle \widehat u'=0&\hbox{on } \omega\times (\widehat\Gamma_0\cup\widehat\Gamma_1),\\
\noame
\displaystyle \widehat u'=0&\hbox{on }  \omega\times T,\\
\displaystyle \widehat \pi\in L^{2}(\omega;L^{2}_{0,\#}(Z'_f)).
\end{array}\right.
\end{equation}
\end{itemize}
\end{theorem}

\begin{proof} The divergence equations in $\omega\times  Z_f$ and $\omega$ in problems (\ref{limit_model_gamma_neq_1}) and (\ref{limit_model_gamma_eq_1}) follow directly from Lemma \ref{lem_conv_vel}. Let us derive the limit model. Consider
$$v_\ep(x',z)=\ep^{1-\gamma}\Big(\varphi(x',z)-\ep^{\gamma-2}\widehat u_\ep(x',z)\Big),$$
as a test function in (\ref{form_var_hat}), with $\varphi\in \mathcal{D}(\omega; C^\infty_{\#}(Z_f)^3)$ with $\varphi_3\equiv 0$, $\varphi'=0$ on $\omega\times T$ and  $\varphi'=0$ on $\omega\times (\widehat \Gamma_0\cup \widehat \Gamma_1)$, satisfying the divergence conditions ${\rm div}_{x'}(\int_{Z_f} \varphi')\,dz=0$ in $\omega$, $(\int_{Z_f}\varphi'\,dz )\cdot n=0$ on $\partial\omega$, and ${\rm div}_{z'}(\varphi')=0$ in $\omega\times Z_f$. Then, considering the definition of $A_r$ given in (\ref{defAr}), we obtain:
\begin{equation*}\begin{array}{l}
\displaystyle
\medskip
  \ep^{\gamma-1}(A_r(\varepsilon^{-1}\widehat { u}_\varepsilon),{ v}_\varepsilon)-\int_{\omega\times Z_f} \widehat p_\varepsilon^0\,{\rm div}_{x'}(v'_\ep)\,dx'dz \\
  \noame\displaystyle 
   -\int_{\omega\times Z_f}\widehat p_\varepsilon^1\,(\varepsilon^{-\ell}{\rm div}_{z'}({ v}'_\ep)+\ep^{-1}\partial_{z_3}v_{\ep,3})\,dx'dz=\int_{\omega\times Z_f} { f}'\cdot { v}'_\varepsilon\,dx'dz+O_\varepsilon,
\end{array}
\end{equation*}
with $|O_\eps|\leq C\eps^{\ell+1-\gamma}$, 
which is equivalent to
\begin{equation*}\begin{array}{l}
\displaystyle
\medskip
 \ep^{\gamma-1}(A_r(\varepsilon^{1-\gamma}\varphi)-A_r(\varepsilon^{-1}\widehat { u}_\varepsilon),{ v}_\varepsilon)- \ep^{\gamma-1} (A_r(\varepsilon^{1-\gamma}\varphi),{ v}_\varepsilon)\\
 \noame
 \displaystyle
+\int_{\omega\times Z_f} \widehat p_\varepsilon^0\,{\rm div}_{x'}(v'_\ep)\,dx'dz
 \displaystyle +\int_{\omega\times Z_f}\widehat p_\varepsilon^1\,(\varepsilon^{-\ell}{\rm div}_{z'}({ v}'_\ep)+\ep^{-1}\partial_{z_3}v_{\ep,3})\,dx'dz
 =-\int_{\omega\times Z_f} { f}'\cdot { v}'_\varepsilon\,dx'dz+O_\varepsilon.
\end{array}
\end{equation*}
Due to (\ref{monotonicity}), we can deduce
\begin{equation*}\begin{array}{l}
\displaystyle
\medskip
 \ep^{\gamma-1} (A_r(\varepsilon^{1-\gamma}\varphi),{ v}_\varepsilon)+\int_{\omega\times Z_f} \widehat p_\varepsilon^0\,{\rm div}_{x'}(v'_\ep)\,dx'dz \\
 \noame
 \displaystyle -\int_{\omega\times Z_f}\varepsilon^{-\ell}\widehat p_\varepsilon^1\,(\varepsilon^{-\ell}{\rm div}_{z'}({ v}'_\ep)+\ep^{-1}\partial_{z_3}v_{\ep,3})\,dx'dz\ge\int_{\omega\times Z_f} { f}'\cdot { v}'_\varepsilon\,dx'dz+O_\varepsilon,
\end{array}
\end{equation*}
Taking into account that ${\rm div}_\ep(\widetilde u_\ep)=0$ implies $\ep^\ell{\rm div}_{z'}(\widehat u_\ep')+\ep^{-1}\partial_{z_3}\widehat u_{\ep,3}=0$,      $\varphi_3\equiv 0$ and  ${\rm div}_{z'}(\varphi')=0$ in $\omega\times Z_f$, we have that
$$\varepsilon^{-\ell}{\rm div}_{z'}(\varepsilon^{1-\gamma}\varphi'-\ep^{-1}\widehat u_\ep')+\ep^{-1}\partial_{z_3}(\varepsilon^{1-\gamma}\varphi_3-\ep^{-1}\widehat u_{\ep,3})=\varepsilon^{1-\ell-\gamma}{\rm div}_{z'}(\varphi')+\ep^{-\ell}{\rm div}_{z'}(\widehat u_\ep')+\ep^{-1}\partial_{z_3}\widehat u_{\ep,3}=0,$$
and taking into account (\ref{productAB}), we get
\begin{equation}\label{Proof_1}\begin{array}{l}
\displaystyle
\medskip
\ep^{1-\gamma}(\eta_0-\eta_\infty)\int_{\omega\times Z_f}(1+\lambda\ep^{2(1-\gamma)}|\varepsilon^{1-\ell} \mathbb{D}_{z'}[  \varphi'] + \partial_{z_3}[ \varphi']|^2)^{{r\over 2}-1}(\varepsilon^{1-\ell} \mathbb{D}_{z'}[  \varphi'] + \partial_{z_3}[ \varphi'])\\
\noame
 \qquad\qquad\qquad\qquad\displaystyle :(\varepsilon^{1-\ell} \mathbb{D}_{z'}[  \varphi'-\ep^{\gamma-2}\widehat u_\ep'] + \partial_{z_3}[ \varphi'-\ep^{\gamma-2}\widehat u_\ep'])dx'dz\\
\noame
\medskip
\displaystyle+  \ep^{1-\gamma}\eta_\infty\int_{\omega\times Z_f}(\varepsilon^{1-\ell} \mathbb{D}_{z'}[  \varphi'] + \partial_{z_3}[ \varphi']):(\varepsilon^{1-\ell} \mathbb{D}_{z'}[  \varphi'-\ep^{\gamma-2}\widehat u_\ep'] + \partial_{z_3}[  \varphi'-\ep^{\gamma-2}\widehat u_\ep'])dx'dz\\
\medskip
\displaystyle+\ep^{1-\gamma}\int_{\omega\times Z_f} \widehat p_\varepsilon^0\,{\rm div}_{x'}(\varphi'-\ep^{\gamma-2}\widehat u_\ep')\,dx'dz  \ge\ep^{1-\gamma}\int_{\omega\times Z_f} { f}'\cdot ( \varphi'-\ep^{\gamma-2}\widehat u_\ep')\,dx'dz+O_\varepsilon,
\end{array}
\end{equation}
where $|O_\ep|\leq C\ep^{\ell+1-\gamma}$. Dividing by $\ep^{1-\gamma}$, we get
\begin{equation}\label{Proof_2}\begin{array}{l}
\displaystyle
\medskip
 (\eta_0-\eta_\infty)\int_{\omega\times Z_f}(1+\lambda\ep^{2(1-\gamma)}|\varepsilon^{1-\ell} \mathbb{D}_{z'}[  \varphi'] + \partial_{z_3}[ \varphi']|^2)^{{r\over 2}-1}(\varepsilon^{1-\ell} \mathbb{D}_{z'}[  \varphi'] + \partial_{z_3}[ \varphi'])\\
\noame
 \qquad\qquad\qquad\qquad\displaystyle :(\varepsilon^{1-\ell} \mathbb{D}_{z'}[  \varphi'-\ep^{\gamma-2}\widehat u_\ep'] + \partial_{z_3}[ \varphi'-\ep^{\gamma-2}\widehat u_\ep'])dx'dz\\
\noame
\medskip
\displaystyle+  \eta_\infty\int_{\omega\times Z_f}(\varepsilon^{1-\ell} \mathbb{D}_{z'}[  \varphi'] + \partial_{z_3}[ \varphi']):(\varepsilon^{1-\ell} \mathbb{D}_{z'}[  \varphi'-\ep^{\gamma-2}\widehat u_\ep'] + \partial_{z_3}[  \varphi'-\ep^{\gamma-2}\widehat u_\ep'])dx'dz\\
\medskip
\displaystyle+\int_{\omega\times Z_f} \widehat p_\varepsilon^0\,{\rm div}_{x'}(\varphi'-\ep^{\gamma-2}\widehat u_\ep')\,dx'dz  \ge \int_{\omega\times Z_f} { f}'\cdot ( \varphi'-\ep^{\gamma-2}\widehat u_\ep')\,dx'dz+O_\varepsilon,
\end{array}
\end{equation}
where $|O_\ep|\leq C\ep^\ell$, which tends to zero when $\ep$ tends to  zero. Now, we can pass to the limit in every term:
\begin{itemize}
\item First and second terms in (\ref{Proof_2}). We divide this into three points, depending on the value of $\gamma$:
\begin{enumerate}
\item  If $\gamma<1$, then $2(1-\gamma)>0$ and so, 
$$\lambda\ep^{2(1-\gamma)}|\varepsilon^{1-\ell} \mathbb{D}_{z'}[  \varphi'] + \partial_{z_3}[ \varphi']|^2\to 0.$$
Thus, from $\ep^{1-\ell}\to0$ and convergence (\ref{conv_vel_hat2}), passing to the limit as $\ep\to 0$,  we find that the first and second terms in (\ref{Proof_2}) converge to 
$$\begin{array}{l}
\displaystyle
\medskip
 (\eta_0-\eta_\infty)\int_{\omega\times Z_f}\partial_{z_3}[ \varphi']:\partial_{z_3}[ \varphi'- \widehat u']\,dx'dz+\eta_\infty \int_{\omega\times Z_f}\partial_{z_3}[ \varphi']:\partial_{z_3}[ \varphi'- \widehat u']\,dx'dz\\
 \noame
 \displaystyle =\eta_0\int_{\omega\times Z_f}\partial_{z_3}[ \varphi']:\partial_{z_3}[ \varphi'- \widehat u']\,dx'dz.
 \end{array}$$
 \item If $\gamma>1$, then $2(1-\gamma)<0$ and so 
$$(1+\lambda\ep^{2(1-\gamma)}|\varepsilon^{1-\ell} \mathbb{D}_{z'}[  \varphi'] + \partial_{z_3}[ \varphi']|^2)^{{r\over 2}-1}\to 0.$$ Hence, the first and second terms in (\ref{Proof_2}) converge to 
$$\eta_\infty\int_{\omega\times Z_f}\partial_{z_3}[ \varphi']:\partial_{z_3}[ \varphi'- \widehat u']\,dx'dz.$$
\item If $\gamma=1$, then 
$$(1+\lambda|\varepsilon^{1-\ell} \mathbb{D}_{z'}[  \varphi'] + \partial_{z_3}[ \varphi']|^2)^{{r\over 2}-1}
\to (1+\lambda|\partial_{z_3}[ \varphi']|^2)^{{r\over 2}-1}.$$
Therefore, we identify the limit as 
$$ (\eta_0-\eta_\infty)\int_{\omega\times Z_f}(1+\lambda |\partial_{z_3}[ \varphi']|^2)^{{r\over 2}-1}\partial_{z_3}[ \varphi']:\partial_{z_3}[ \varphi'- \widehat u'])dx'dz+\eta_\infty\int_{\omega\times Z_f}\partial_{z_3}[ \varphi']:\partial_{z_3}[ \varphi'- \widehat u']\,dx'dz.$$
\end{enumerate}

\item Third term in (\ref{Proof_2}). By using the strong convergence (\ref{pressuresq})$_2$ of $\widehat p_\ep^0$ and the weak convergence (\ref{conv_vel_hat2}) of $\widehat u_\ep$, passing to the limit as $\ep\to 0$ and using divergence condition ${\rm div}_{x'}(\int_{Z_f} \varphi')\,dz={\rm div}_{x'}(\int_{Z_f} \widehat u')\,dz=0$ , the third term in (\ref{Proof_2}) converges to 
$$\begin{array}{l}\displaystyle
\int_{\omega\times Z_f} \widetilde p(x')\,{\rm div}_{x'}(\varphi'-\widehat u')\,dx'dz=\int_{\omega}\widetilde p(x'){\rm div}_{x'}\left(\int_{Z_f}(\varphi'-\widehat u')\,dz\right)dx'=0.
\end{array}$$
\item Last term in (\ref{Proof_2}). Using convergence  (\ref{conv_vel_hat2}), the last term converges to
$$\int_{\omega\times Z_f} { f}'\cdot ( \varphi'- \widehat u')\,dx'dz.$$
\end{itemize}
In summary, considering $\eta$ equal to $\eta_0$ if $\gamma<1$ or $\eta_\infty$ if $\gamma>1$, we obtain the following inequality:
$$ \eta\int_{\omega\times Z_f}\partial_{z_3}[ \varphi']:\partial_{z_3}[ \varphi'- \widehat u']\,dx'dz\geq\int_{\omega\times Z_f}f'\cdot (\varphi'- \widehat u')\,dx'dz,$$
and for $\gamma=1$:
$$\begin{array}{l}\displaystyle(\eta_0-\eta_\infty)\int_{\omega\times Z_f}(1+\lambda |\partial_{z_3}[ \varphi']|^2)^{{r\over 2}-1}\partial_{z_3}[ \varphi']:\partial_{z_3}[ \varphi'- \widehat u'])dx'dz+\eta_\infty\int_{\omega\times Z_f}\partial_{z_3}[ \varphi']:\partial_{z_3}[ \varphi'- \widehat u']\,dx'dz\\
\noame
\displaystyle\geq \int_{\omega\times Z_f}f'\cdot (\varphi'- \widehat u')\,dx'dz.
\end{array}$$
Because $\varphi'$ is arbitrary, using Minty's lemma (see \cite[Chapter 3, Lemma 1.2]{Lions2}), we deduce that the following equality holds true for $\gamma\neq 1$:
\begin{equation}\label{problem_limit1} \eta\int_{\omega\times Z_f}\partial_{z_3}[ \widehat u']:\partial_{z_3}[ v']\,dx'dz=\int_{\omega\times Z_f}f'\cdot v'\,dx'dz,
\end{equation}
and for  $\gamma=1$, the following one:
\begin{equation}\label{problem_limit2}\begin{array}{l}\displaystyle(\eta_0-\eta_\infty)\int_{\omega\times Z_f}(1+\lambda |\partial_{z_3}[ \widehat u']|^2)^{{r\over 2}-1}\partial_{z_3}[ \widehat u']:\partial_{z_3}[ v'])dx'dz+\eta_\infty\int_{\omega\times Z_f}\partial_{z_3}[ \widehat u']:\partial_{z_3}[ v']\,dx'dz\\
\noame
\displaystyle=\int_{\omega\times Z_f}f'\cdot v'\,dx'dz,
\end{array}
\end{equation}
which by density are satisfied by every $v'$ in the Hilbert space $\mathbb{V}_2$ defined by
\begin{equation}\label{Vspace2}
\mathbb{V}_2=\left\{\begin{array}{l}
v'(x',z)\in V_{z_3,\#}^2(\omega\times Z_f)^2\quad\hbox{such that } \\
\noame
\displaystyle
{\rm div}_{x'}\left(\int_{Z_f}v'(x',z)\,dz\right)=0\quad\hbox{in }\omega,\quad \left(\int_{Z_f}v'(x',z)\,dz\right)\cdot n=0\quad\hbox{on }\partial\omega,\\
\noame
\displaystyle
{\rm div}_{z'}(v')=0\quad\hbox{in }\omega\times Z_f,\quad v'(x',z)=0\quad\hbox{on }(\omega\times T)\hbox{ and }\omega \times (\widehat \Gamma_0\cup\widehat \Gamma_1)
\end{array}\right\}.
\end{equation}
Reasoning as in \cite{Anguiano_Bonn_SG_Carreau2}, and taking into account the relations
\begin{equation}\label{equatlitysimmetruic}\partial_{z_3}[\widehat {  u}']: \partial_{z_3} [{  v}']={1\over 2}\partial_{z_3}\widehat {  u}' \cdot \partial_{z_3} { v}'\quad \hbox{and}\quad |\partial_{z_3}[\widehat { u}']|^2={1\over 2}|\partial_{z_3}\widehat {  u}'|^2,
\end{equation} 
 we can deduce that the limit variational formulation (\ref{problem_limit1}) is equivalent to problem (\ref{limit_model_gamma_neq_1})  for $\gamma\neq 1$, where $\eta$ is equal to $\eta_0$ if $\gamma<1$ or $\eta_\infty$ if $\gamma>1$, and that the variational formulation  (\ref{problem_limit2}) is equivalent to problem (\ref{limit_model_gamma_eq_1}) for $\gamma=1$. 
 
According to \cite{BayadaChambat} and \cite[Propositions 3.2 and 3.3]{Tapiero2}, it can be proven that (\ref{problem_limit1}) and (\ref{problem_limit2}) have a unique solution  $(\widehat u', \widehat \pi, \widetilde p)\in V_{z_3,\#}^2(\omega\times Z)^2\times L^2(\omega;L^2_{0,\#}(Z'_f))\times (L^{2}_0(\omega)\cap H^1(\omega)$).  Hence the entire sequence  converges.

\end{proof}

We are now in a position to deduce Theorem~\ref{expresiones_finales} from Theorem~\ref{thm_pseudoplastic}.

\emph{Proof of Theorem~\ref{expresiones_finales}.} The proof is divided into four steps. In the first two steps, we consider the case $\gamma=1$ and in the last two steps, we consider the case $\gamma\neq 1$. \\

{\it Step 1. The local  problem for $\gamma=1$.}  Let $q=\max\{2,r\}$ and $q'$ the conjugate exponent of $q$. For every $\delta'\in\mathbb{R}^2$, we consider $(\widehat { w}_{\delta'}(z), \widehat q_{\delta'}(z'))$  satisfying the local problem given by
\begin{equation}\label{limit_model_local}
\left\{\begin{array}{rl}
\displaystyle
-{1\over 2}\partial_{z_3}\left(  \left((\eta_0-\eta_\infty)\left(1+{\lambda\over 2}|\partial_{z_3} \widehat { w}_{\delta'}|^2\right)^{{r\over 2}-1}+\eta_\infty\right)\partial_{z_3}\widehat { w}_{\delta'}\right)+\nabla_{z'}\widehat q_{\delta'}(z')=-\delta' &\hbox{in }  Z_f,\\
\noame
\displaystyle
{\rm div}_{z'}(\widehat { w}_{\delta'})=0&\hbox{in }  Z_f,\\
\noame
\displaystyle \widehat { w}_{\delta'}=0&\hbox{on }  T,\\
\noame
\displaystyle \widehat { w}_{\delta'}=0&\hbox{on } \widehat\Gamma_0\cup\widehat\Gamma_1,\\
\noame
\displaystyle \widehat q_{\delta'}\in L^{q'}_{0,\#}(Z'_f).
\end{array}\right.
\end{equation}
This problem has a unique solution $(\widehat { w}_{\delta'}(z), \widehat q_{\delta'}(z'))\in V_{z_3,\#}^q(Z)^2\times (L^{q'}_{0,\#}(Z'_f)\cap W^{1,q'}(Z'_f))$ for every $\delta'\in\mathbb{R}^2$ (this can be proven by adapting the proofs of \cite[Propositions 3.2 and 3.3]{Tapiero2}).   In fact, the velocity $\widehat { w}_{\delta'}$ is given by
\begin{equation}\label{expresionw}
\widehat {w}_{\delta'}(z)=-2\int_{{1\over 2}-z_3}^{1\over 2}{\xi\over \psi(2|\delta'+\nabla_{z'}\widehat q_{\delta'}(z')||\xi|)}d\xi\left(\delta'+\nabla_{z'}\widehat q_{\delta'}(z')\right),
\end{equation}
where $\psi$ is the inverse function of (\ref{taupsi}), which has a unique solution noted by $\zeta=\psi(\tau)$ for $\tau\in\mathbb{R}^+$ (see \cite[Proposition 3.3]{Tapiero2} for more details).

To prove this, we write (\ref{limit_model_local})$_1$ as follows
\begin{equation}\label{eqequiv}-\partial_{z_3}(\eta_r(\partial_{z_3}\widehat {w}_{\delta'})\partial_{z_3}\widehat {w}_{\delta'})=-2(\delta'+\nabla_{z'}\widehat q_{\delta'}(z')), \quad \eta_r(\xi')=(\eta_0-\eta_\infty)\left(1+{\lambda\over 2}|\xi'|^2\right)^{{r\over 2}-1}+\eta_\infty,\quad \xi'\in \mathbb{R}^2,
\end{equation}
we set $g(z')=-2(\delta'+\nabla_{z'}\widehat q_{\delta'}(z'))$ and integrating (\ref{eqequiv}) with respect to $z_3$, we deduce
\begin{equation}\label{eqequiv2}
\eta_r(\partial_{z_3}\widehat {w}_{\delta'})\partial_{z_3}\widehat {w}_{\delta'}=C(z')-z_3 g(z').
\end{equation}
Set $\Gamma(z)=C(z')-z_3 g(z')$ and $y=|\partial_{z_3}\widehat {w}_{\delta'}|$ and so,  by (\ref{eqequiv}) we have that
$$\left(1+{\lambda\over 2}y^2\right)^{{r\over 2}-1}={\eta_r(y)-\eta_\infty\over \eta_0-\eta_\infty},$$
i.e.,
\begin{equation}\label{eqequiv3}y=\sqrt{{2\over \lambda}\left\{\left({\eta_r(y)-\eta_\infty\over \eta_0-\eta_\infty}\right)^{2\over r-2}
-1\right\}}.
\end{equation}
Putting (\ref{eqequiv3}) in (\ref{eqequiv2}), we obtain 
$$\eta_r(y)\sqrt{{2\over \lambda}\left\{\left({\eta_r(y)-\eta_\infty\over \eta_0-\eta_\infty}\right)^{2\over r-2}
-1\right\}}=|\Gamma|,$$
which is the same equation as \eqref{taupsi} with $\zeta=\eta_r(y)$ and $\tau=|\Gamma|$.  Then,  we have $\eta_r(y)=\psi(|\Gamma|)$ and with (\ref{eqequiv2}), we deduce that:
$$\partial_{z_3}\widehat {w}_{\delta'}(z)={\Gamma(z)\over \psi(|\Gamma(z)|)}.$$ Integrating with respect to the vertical variable   from $0$ to $z_3$ we get
\begin{equation}\label{uexpresion}
\widehat {w}_{\delta'}(z)=\int_0^{z_3}{\Gamma(z',\zeta)\over \psi(|\Gamma(z',\zeta)|)}\,d\zeta=\int_0^{z_3}{C(z')-\zeta g(z')\over \psi(|C(z')-\zeta g(z')|)}\,d\zeta.
\end{equation}

As a consequence of the proof of uniqueness of the solution to~\eqref{limit_model_local} (see~\cite{Tapiero2}), since $\widehat {w}_{\delta'}$ vanishes on $z_3=1$, the constant $C(z')$ appearing in~\eqref{eqequiv2} is the only constant such that 
\begin{equation}\label{eqequiv4}\int_0^{1}{C(z')-\zeta g(z')\over \psi(|C(z')-\zeta g(z')|)}\,d\zeta=0.
\end{equation}
Hence, one has $C(z')={g(z')\over 2}$ because using this value of $C(z')$ and the change of variables $\xi={1\over 2}-\zeta$ in (\ref{eqequiv4}), we obtain
$$\int_0^{1}{g(z')({1\over 2}-\zeta)\over \psi(|g(z')||{1\over 2}-\zeta|)}\,d\zeta=g(z')\int_{-{1\over 2}}^{1\over 2}{\xi\over \psi(|g(z')||\xi|)}\,d\xi=0.$$
Therefore, we have obtained the expression
$$\widehat w_{\delta'}(z)=g(z')\int_{0}^{z_3}{{1\over 2}-\zeta\over \psi(|g(z')||{1\over 2}-\zeta|)}\,d\zeta=g(z')\int_{{1\over 2}-z_3}^{1\over 2}{\xi\over \psi(|g(z')||\xi|)}\,d\xi.$$
Finally, integrating $\widehat {w}_{\delta'}$ with respect to $z_3$ between $0$ and $1$ yields
$$\begin{array}{rl}\displaystyle 
 \int_0^1\widehat {w}_{\delta'}(z)\,dz_3= &\displaystyle g(z')\int_0^{1}\left(\int_{{1\over 2}-z_3}^{1\over 2}{\xi\over \psi(|g(y')||\xi|)}\,d\xi\right)\,dz_3\,\\
\noame
=&\displaystyle g(z')\int_{-{1\over 2}}^{1\over 2}{\xi\over \psi(|g(z')||\xi|)}\left(\int_{{1\over 2}-\xi}^{1}\,dy_z\right)\,d\xi,
\end{array}$$
that is., expression 
\begin{equation}\label{meanwhat}\ \int_0^1\widehat {w}_{\delta'}(z)\,dz_3=\displaystyle -2(\delta'+\nabla_{z'}\widehat q_{\delta'}(z'))\int_{-{1\over 2}}^{1\over 2}{({1\over 2}+\xi)\xi \over \psi(1|\delta'+\nabla_{z'}\widehat q_{\delta'}(z')||\xi|)}d\xi.
\end{equation}
 Because of the divergence-free condition ${\rm div}_{z'}(\int_{0}^{1}\widehat { w}_{\delta'}\,dz_3)=0$ in $Z'$, the pressure $\widehat q_{\delta'}$ satisfies limit law (\ref{expressionq}).\\

{\it Step 2. Identification of the filtration velocity (\ref{FiltrationVelcovcity}) and the Darcy problem (\ref{Reynolds_equation}).} Using the idea from \cite{Bourgeat1} to decouple the homogenized problem of Carreau type (\ref{limit_model_gamma_eq_1}), for every $\delta'\in\mathbb{R}^2$ we consider the function $\mathcal{U}:\mathbb{R}^2\to \mathbb{R}^2$ given by
\begin{equation}\label{permeabilityU}\mathcal{U}(\delta')=\int_{Z_f} \widehat { w}_{\delta'}(z)\,dz,
\end{equation}
where $\widehat {w}_{\delta'}(z)$ is the solution to system (\ref{limit_model_local}). Thus, $(\widehat { u}',\widehat \pi)$ takes the form
$$\widehat{ u}'(x',z)=\widehat { w}_{\nabla_{x'}\widetilde p(x')-{ f}'(x')}(z),\quad \widehat \pi(x',z)=\widehat q_{\nabla_{x'}\widetilde p(x')-{ f}'(x')}(z)\quad \hbox{in }\omega\times Z_f,$$
and then, from the relation $\widetilde V'(x')=\int_0^{1}\widetilde { u}'(x',z_3)\,dz_3=\int_{Z_f}\widehat { u}'(x',z)\,dz$ given in (\ref{relation2}), and considering that  $\widetilde V_3(x')=\int_0^{1}\widetilde u_3(x',z_3)\,dz_3=\int_{Z_f}\widehat u_3\,dz=0$, we deduce the filtration velocity:
$$\widetilde V'(x')=\int_0^{1}\widetilde u'(x',z_3)\,dz_3=\mathcal{U}(\nabla_{x'}\widetilde p(x')-{ f}'(x')),\quad \widetilde V_3\equiv 0,\quad \hbox{in }\omega.$$
Then, by using (\ref{permeabilityU}) together with (\ref{meanwhat}) with $\delta'=\nabla_{x'}\widetilde p(x')-{ f}'(x')$,  we deduce the expression for the filtration velocity (\ref{FiltrationVelcovcity}).

Finally, from conditions (\ref{limit_model_gamma_eq_1})$_{3,4}$, we deduce the Darcy problem for pressure (\ref{Reynolds_equation}).\\

{\it Step 3. The local problem for $\gamma\neq 1$.} We define  the following local problems for $\widehat w^i(z)\in V_{z_3,\#}^2(Z_f)$, $i=1,2$, the unique solution of
  \begin{equation}\label{local_problems_1}
\left\{\begin{array}{rl}
\displaystyle
-{\eta\over 2}\partial_{z_3}^2\widehat w^i+\nabla_{z'}\widehat q^i(z')=-e_i &\hbox{in }  Z_f,\\
\noame
\displaystyle
{\rm div}_{z'}(\widehat w^i)=0&\hbox{in }  Z_f,\\
\noame
\displaystyle \  \widehat w^i=0&\hbox{on } T,\\
\noame
\displaystyle \  \widehat w^i=0&\hbox{on } \widehat \Gamma_0\cup \widehat \Gamma_1,\\
\noame
\displaystyle \widehat \pi^i\in L^{2}_{0,\#}(Z'_f),
\end{array}\right.
\end{equation}
where $\{e_1,e_2\}$ is the canonical basis of $\mathbb{R}^2$. 
However, because local problems (\ref{local_problems_1}) can be solved, by means of integration and boundary conditions on $z_3=\{0,1\}$, we obtain:
$$\widehat w^i(z)={1\over \eta}(e_i+\nabla_{z'}\widehat q^i(z'))(z_3^2-z_3),$$
and so
\begin{equation}\label{meanw}\int_0^1\widehat w^i(z)\,dz_3=-{1\over 6\eta}(e_i+\nabla_{z'}\widehat q^i(z')).
\end{equation}
 By using the condition ${\rm div}_{z'}\left(\int_0^1\widehat w^i(z)dz_3\right)=0$ in $Z'_f$, $\widehat w^i$ can be explicitly obtained by means of functions $ \widehat q^i(z')$, $i=1,2$, which are the unique solution in $H^1_\#(Z'_f)$ of the local problem (\ref{local_problems_0}).\\

{\it Step 4. Identification of filtration velocities (\ref{FiltrationVelcovcity0}) and (\ref{Reynolds_equation0})}.  Following  the proof of Theorem 2.1-$(iii)$ in \cite{Anguiano_SG_Transition}, using identification
$$\widehat u(x',z)=\sum_{i=1}^2\left(f_i(x')-\partial_{x_i}\widetilde p(x')\right)\widehat w^i(z),\quad \widehat \pi(x',z)=\sum_{i=1}^2\left(f_i(x')-\partial_{x_i}\widetilde p(x')\right)\widehat q^i(z),$$
with $(\widehat w^i, \widehat q^i)$, $i=1,2$, the unique  solution of (\ref{local_problems_1}), 
we can deduce that 
$$\widetilde V'(x')= A(f'(x')-\nabla_{x'}\widetilde p(x')),\quad \widetilde V_3(x')\equiv 0,$$
where $A\in \mathbb{R}^{2\times 2}$ is symmetric and definite positive defined by its entries
$$A_{ij}=\int_{Z_f}\partial_{z_3}\widehat w^i(z)\cdot \partial_{z_3}\widehat w^j(z)\,dz=\int_{Z_f}\widehat w^i(z)\cdot e_j\,dz,\quad i=1, 2,\  j=1,2.$$
However, because $\widehat w^i$ can be expressed by means of functions $\widehat q^i$ (as seen in Step 3), we have
$$A_{ij}=\int_{Z_f}\widehat w^i(z)\cdot e_j\,dz=\int_{Z'_f}\left(\int_0^1\widehat w^i(z)\,dz_3\right)\cdot e_j\,dz',\quad i=1, 2,\  j=1,2.$$
Thus, using expression (\ref{meanw}), we can deduce expressions (\ref{FiltrationVelcovcity0}) for $\widetilde V$ and (\ref{permeabilityA0}) for $A$ where $\widehat q^i$, $i=1,2$ denotes the unique solution of (\ref{local_problems_0}).

  Finally, from the divergence conditions in (\ref{limit_model_gamma_neq_1}) and the expression for $\widetilde V'$, we deduce the equation~(\ref{Reynolds_equation0}) for the pressure, which is well-posed,
because it is simply a second-order elliptic equation.

\qed

\subsection{Limit models in case $r>2$ (dilatant fluids)}
\begin{theorem}[Case $r>2$]  \label{thm_dilatant} For $r>2$, the pair $(\widehat u, \widetilde p)$  given in Lemmas \ref{lem_conv_vel} and \ref{lem_conv_press} is the unique solution to a reduced two-pressure problem, which depends on the value of $\gamma$.
\begin{itemize}
\item If $\gamma< 1$, then $(\widehat u, \widetilde p)\in V_{z_3,\#}^2(\omega\times Z_f)^2\times (L^{2}_0(\omega)\cap H^1(\omega))$ satisfies problem (\ref{limit_model_gamma_neq_1}) with $\eta=\eta_0$.
 
 \item If $\gamma > 1$, then $(\widehat u, \widetilde p)\in V_{z_3,\#}^r(\omega\times Z_f)^2\times (L^{r'}_0(\omega)\cap W^{1,r'}(\omega))$ satisfies  the following reduced power-law Stokes problem
\begin{equation}\label{limit_model_gamma_sup_1_r}
\left\{\begin{array}{rl}
\displaystyle
-2^{-{r\over 2}}(\eta_0-\eta_\infty)\lambda^{r-2\over 2}\partial_{z_3}(|\partial_{z_3}\widehat u'|^{r-2}\partial_{z_3}\widehat { u}')  +\nabla_{z'}\widehat \pi(z')={ f}'(x')-\nabla_{x'}\widetilde p(x') &\hbox{in }\omega\times Z_f,\\
\noame
\displaystyle
{\rm div}_{z'}(\widehat { u}')=0&\hbox{in }\omega\times Z_f,\\
\noame
\displaystyle
{\rm div}_{x'}\left(\int_{Z_f}\widehat u'\,dz\right)=0&\hbox{in }\omega,\\
\noame
\displaystyle \left(\int_{Z_f}\widehat u'\,dz\right)\cdot n=0&\hbox{on }\partial\omega,
\\
\noame
\displaystyle \widehat u'=0&\hbox{on }  \omega\times T,\\
\noame
\displaystyle \widehat u'=0&\hbox{on } \omega\times (\widehat\Gamma_0\cup\widehat\Gamma_1),\\
\noame
\displaystyle \widehat \pi\in L^{r'}(\omega;L^{r'}_{0,\#}(Z'_f)).
\end{array}\right.
\end{equation}

\item If $\gamma = 1$, then $(\widehat u, \widetilde p)\in V_{z_3,\#}^r(\omega\times Z_f)^2\times (L^{r'}_0(\omega)\cap W^{1,r'}(\omega))$ satisfies  problem (\ref{limit_model_gamma_eq_1}).

\end{itemize}
\end{theorem}
\begin{proof}
We divide the proof into three steps. 
\\

{\it Step 1. Case $\gamma<1$}. The proof of (\ref{limit_model_gamma_neq_1}) for $r>2$ and $\gamma<1$ is similar to the proof of Theorem \ref{thm_pseudoplastic} in case $1<r<2$ and $\gamma\neq 1$. We present only the main steps:
\begin{enumerate}
\item We deduce the variational inequality (\ref{Proof_2}) for every  $\varphi\in \mathcal{D}(\omega; C^\infty_{\#}(Z_f)^3)$ with $\varphi_3\equiv 0$, $\varphi'=0$ on $\omega\times T$ and  $\varphi'=0$ on $\omega\times (\widehat \Gamma_0\cup \widehat \Gamma_1)$, satisfying the divergence conditions ${\rm div}_{x'}(\int_{Z_f} \varphi')\,dz=0$ in $\omega$, $(\int_{Z_f}\varphi'\,dz )\cdot n=0$ on $\partial\omega$, and ${\rm div}_{z'}(\varphi')=0$ in $\omega\times Z_f$. Here, $|O_\ep|\leq C\ep^\ell$, which tends to zero when $\ep$ tends to zero.

\item We can pass to the limit in every term of (\ref{Proof_2}) when $\ep$ tends to  zero. The proof is similar to that given for case $1<r<2$ and $\gamma\neq 1$,  considering convergences~(\ref{conv_vel_tilde2}) and (\ref{conv_vel_hat2}). We notice that because $\gamma<1$, $\lambda \ep^{2(1-\gamma)}$ tends to zero, and thus, since $r>2$, it holds that 
$$(1+\lambda\ep^{2(1-\gamma)}|\varepsilon^{1-\ell} \mathbb{D}_{z'}[  \varphi'] + \partial_{z_3}[ \varphi']|^2)^{{r\over 2}-1}\to 1.
$$
As a result, the first and second terms of (\ref{Proof_2}) converge to 
$$\eta_0\int_{\omega\times Z_f}\partial_{z_3}[\varphi']:\partial_{z_3}[\varphi'-\widehat u']\,dx'dz.$$

\item We conclude that $(\widehat u, \widetilde p)$ satisfies the variational formulation 
$$\eta_0\int_{\omega\times Z_f}\partial_{z_3}[\widehat u']:\partial_{z_3}[v']\,dx'dz=\int_{\omega\times Z_f}f'\cdot v'\,dx'dz,$$
for every function $v'$ in Hilbert space $\mathbb{V}_2$ defined by (\ref{Vspace2}). This variational formulation is equivalent to problem (\ref{limit_model_gamma_neq_1}) with $\eta=\eta_0$.

\end{enumerate}

{\it Step 2. Case $\gamma>1$}. Conditions (\ref{limit_model_gamma_sup_1_r})$_{2,3,4,5}$ follow from Lemma \ref{lem_conv_vel}. To prove that $(\widehat u, \widetilde p)$ satisfies the momentum equation (\ref{limit_model_gamma_sup_1_r})$_{1}$, we follow the lines of the proof to obtain (\ref{Proof_2}), but choose now $v_\ep$ and $\varphi$ such that 
$$v_\ep=\ep^{1-\gamma\over r-1}\Big(\varphi-\ep^{\gamma-r\over r-1}\widehat u_\ep\Big),$$
 with $\varphi\in \mathcal{D}(\omega; C^\infty_{\#}(Z_f)^3)$ with $\varphi_3\equiv 0$, $\varphi'=0$ on $\omega\times T$ and  $\varphi'=0$ on $\omega\times (\widehat \Gamma_0\cup \widehat \Gamma_1)$, satisfying the divergence conditions ${\rm div}_{x'}(\int_{Z_f} \varphi')\,dz=0$ in $\omega$, $(\int_{Z_f}\varphi'\,dz )\cdot n=0$ on $\partial\omega$, and ${\rm div}_{z'}(\varphi')=0$ in $\omega\times Z_f$. Then, we get
\begin{equation}\label{Proof_1_super}\begin{array}{l}
\displaystyle
\medskip
\ep^{\gamma-1+2{1-\gamma\over r-1}}(\eta_0-\eta_\infty)\int_{\omega\times Z_f}(1+\lambda\ep^{2{1-\gamma\over r-1}}|\varepsilon^{1-\ell} \mathbb{D}_{z'}[  \varphi'] + \partial_{z_3}[ \varphi']|^2)^{{r\over 2}-1}(\varepsilon^{1-\ell} \mathbb{D}_{z'}[  \varphi'] + \partial_{z_3}[ \varphi'])\\
\noame
 \qquad\qquad\qquad\qquad\displaystyle :(\varepsilon^{1-\ell} \mathbb{D}_{z'}[  \varphi'-\ep^{{\gamma-r\over r-1}}\widehat u_\ep'] + \partial_{z_3}[ \varphi'-\ep^{\gamma-r\over r-1}\widehat u_\ep'])dx'dz\\
\noame
\medskip
\displaystyle+ \ep^{\gamma-1+2{1-\gamma\over r-1}}\eta_\infty\int_{\omega\times Z_f}(\varepsilon^{1-\ell} \mathbb{D}_{z'}[  \varphi'] + \partial_{z_3}[ \varphi']):(\varepsilon^{1-\ell} \mathbb{D}_{z'}[   \varphi'-\ep^{{\gamma-r\over r-1}}\widehat u_\ep'] + \partial_{z_3}[  \varphi'-\ep^{{\gamma-r\over r-1}}\widehat u_\ep'])dx'dz\\
\medskip
\displaystyle+\ep^{1-\gamma\over r-1}\int_{\omega\times Z_f} \widehat p_\varepsilon^0\,{\rm div}_{x'}( \varphi'-\ep^{{\gamma-r\over r-1}}\widehat u_\ep')\,dx'dz  \ge\ep^{1-\gamma\over r-1}\int_{\omega\times Z_f} { f}'\cdot (  \varphi'-\ep^{{\gamma-r\over r-1}}\widehat u_\ep')\,dx'dz+O_\varepsilon,
\end{array}
\end{equation}
where $|O_\ep|\leq C\ep^{\ell+{1-\gamma\over r-1}}$. Dividing by $\ep^{1-\gamma\over r-1}$, we deduce the inequality  
\begin{equation}\label{Proof_2_super}\begin{array}{l}
\displaystyle
\medskip
\ep^{(\gamma-1){r-2\over r-1}}(\eta_0-\eta_\infty)\int_{\omega\times Z_f}(1+\lambda\ep^{2{1-\gamma\over r-1}}|\varepsilon^{1-\ell} \mathbb{D}_{z'}[  \varphi'] + \partial_{z_3}[ \varphi']|^2)^{{r\over 2}-1}(\varepsilon^{1-\ell} \mathbb{D}_{z'}[  \varphi'] + \partial_{z_3}[ \varphi'])\\
\noame
 \qquad\qquad\qquad\qquad\displaystyle :(\varepsilon^{1-\ell} \mathbb{D}_{z'}[  \varphi'-\ep^{{\gamma-r\over r-1}}\widehat u_\ep'] + \partial_{z_3}[ \varphi'-\ep^{\gamma-r\over r-1}\widehat u_\ep'])dx'dz\\
\noame
\medskip
\displaystyle+ \ep^{(\gamma-1){r-2\over r-1}}\eta_\infty\int_{\omega\times Z_f}(\varepsilon^{1-\ell} \mathbb{D}_{z'}[  \varphi'] + \partial_{z_3}[ \varphi']):(\varepsilon^{1-\ell} \mathbb{D}_{z'}[   \varphi'-\ep^{{\gamma-r\over r-1}}\widehat u_\ep'] + \partial_{z_3}[  \varphi'-\ep^{{\gamma-r\over r-1}}\widehat u_\ep'])dx'dz\\
\medskip
\displaystyle+ \int_{\omega\times Z_f} \widehat p_\varepsilon^0\,{\rm div}_{x'}( \varphi'-\ep^{{\gamma-r\over r-1}}\widehat u_\ep')\,dx'dz  \ge \int_{\omega\times Z_f} { f}'\cdot (  \varphi'-\ep^{{\gamma-r\over r-1}}\widehat u_\ep')\,dx'dz+O_\varepsilon,
\end{array}
\end{equation}
where $|O_\ep|\leq C\ep^\ell$, which tends to zero when $\ep$ tends to zero. Because $\gamma>1$ and $r>2$, we have $(\gamma-1){r-2\over r-1}>0$ and from 
$$
(1+\lambda\ep^{2{1-\gamma\over r-1}}|\varepsilon^{1-\ell} \mathbb{D}_{z'}[  \varphi'] + \partial_{z_3}[ \varphi']|^2)^{{r\over 2}-1}=(\ep^{2{\gamma-1\over r-1}}+\lambda|\varepsilon^{1-\ell} \mathbb{D}_{z'}[  \varphi'] + \partial_{z_3}[ \varphi']|^2)^{{r\over 2}-1},
$$
and convergences (\ref{conv_vel_tilde_r_super}) and (\ref{conv_vel_hat_r_super}), passing to the limit in (\ref{Proof_1_super}) when $\ep$ tends to zero, we deduce that
\begin{equation}\label{Proof_3_super}\begin{array}{l}
\displaystyle
\medskip
 (\eta_0-\eta_\infty)\lambda^{r-2\over 2}\int_{\omega\times Z_f}  | \partial_{z_3}[ \varphi']|^{r-2}\partial_{z_3}[ \varphi'] : \partial_{z_3}[ \varphi'-\widehat u']\,dx'dz\\
\noame
\displaystyle+ \int_{\omega\times Z_f} \widetilde p\,{\rm div}_{x'}( \varphi'-\widehat u')\,dx'dz  \ge \int_{\omega\times Z_f} { f}'\cdot (  \varphi'- \widehat u')dx'dz.
\end{array}
\end{equation}
Because $\widetilde p$ does not depend on $z$, from the divergence conditions ${\rm div}_{x'}\int_{Z_f}\varphi'dz=0$ and ${\rm div}_{x'}\int_{Z_f}\widehat u'dz=0$, we deduce that
$$\int_{\omega\times Z_f} \widetilde p\,{\rm div}_{x'}( \varphi'-\widehat u')\,dx'dz=0.$$
As a result, we obtain the following variational inequality:
\begin{equation}\label{Proof_3_super}\begin{array}{l}
\displaystyle
\medskip
 (\eta_0-\eta_\infty)\lambda^{r-2\over 2}\int_{\omega\times Z_f}  | \partial_{z_3}[ \varphi']|^{r-2}\partial_{z_3}[ \varphi'] : \partial_{z_3}[ \varphi'-\widehat u']\,dx'dz  \ge \int_{\omega\times Z_f} { f}'\cdot (  \varphi'- \widehat u')dx'dz.
\end{array}
\end{equation}
Using Minty's lemma and a density argument, we conclude that the equality 
\begin{equation}\label{Proof_4_super}\begin{array}{l}
\displaystyle
\medskip
 (\eta_0-\eta_\infty)\lambda^{r-2\over 2}\int_{\omega\times Z_f}  | \partial_{z_3}[ \widehat u]|^{r-2}\partial_{z_3}[ \widehat u'] : \partial_{z_3}[ v']\,dx'dz  = \int_{\omega\times Z_f} { f}'\cdot v'\,dx'dz,
\end{array}
\end{equation}
is valid for every $v'$ in the Banach space $\mathbb{V}_r$ defined by
\begin{equation}\label{Vspacer}
\mathbb{V}_r=\left\{\begin{array}{l}
v'(x',z)\in V_{z_3,\#}^r(\omega\times Z_f)^2\quad\hbox{such that } \\
\noame
\displaystyle
{\rm div}_{x'}\left(\int_{Z_f}v'(x',z)\,dz\right)=0\quad\hbox{in }\omega,\quad \left(\int_{Z_f}v'(x',z)\,dz\right)\cdot n=0\quad\hbox{on }\partial\omega,\\
\noame
\displaystyle
{\rm div}_{z'}(v')=0\quad\hbox{in }\omega\times Z_f,\quad v'(x',z)=0\quad\hbox{in }(\omega\times T)\hbox{ and }\omega\times (\widehat \Gamma_0\cup\widehat \Gamma_1)
\end{array}\right\}.
\end{equation}
Reasoning as in \cite{Anguiano_Bonn_SG_Carreau2}, and taking into account (\ref{equatlitysimmetruic}), we can deduce that the limit variational formulation (\ref{Proof_4_super}) is equivalent to problem (\ref{limit_model_gamma_sup_1_r}), which by \cite{Tapiero}, it can be proven that it admits a unique solution $(\widehat u', \widehat \pi, \widetilde p)\in V_{z_3,\#}^r(\omega\times Z_f)^2\times L^{r'}(\omega;L^{r'}_{0,\#}(Z'_f))\times (L^{r'}_0(\omega)\cap W^{1,r'}(\omega))$, hence the entire sequence $(\widehat u_\ep, \widehat p_\ep^0)$ converges to $(\widehat u', \widetilde p)$.
\\

{\it Step 3. Case $\gamma=1$}.  Proceeding similarly to Step 2, but taking $\gamma=1$, we obtain:
\begin{equation}\label{Proof_1_super_eq}\begin{array}{l}
\displaystyle
\medskip
 (\eta_0-\eta_\infty)\int_{\omega\times Z_f}(1+\lambda|\varepsilon^{1-\ell} \mathbb{D}_{z'}[  \varphi'] + \partial_{z_3}[ \varphi']|^2)^{{r\over 2}-1}(\varepsilon^{1-\ell} \mathbb{D}_{z'}[  \varphi'] + \partial_{z_3}[ \varphi'])\\
\noame
 \qquad\qquad\qquad\qquad\displaystyle :(\varepsilon^{1-\ell} \mathbb{D}_{z'}[  \varphi'-\ep^{-1}\widehat u_\ep'] + \partial_{z_3}[ \varphi'-\ep^{-1}\widehat u_\ep'])dx'dz\\
\noame
\medskip
\displaystyle+  \eta_\infty\int_{\omega\times Z_f}(\varepsilon^{1-\ell} \mathbb{D}_{z'}[  \varphi'] + \partial_{z_3}[ \varphi']):(\varepsilon^{1-\ell} \mathbb{D}_{z'}[   \varphi'-\ep^{-1}\widehat u_\ep'] + \partial_{z_3}[  \varphi'-\ep^{-1}\widehat u_\ep'])dx'dz\\
\medskip
\displaystyle+ \int_{\omega\times Z_f} \widehat p_\varepsilon^0\,{\rm div}_{x'}( \varphi'-\ep^{-1}\widehat u_\ep')\,dx'dz \ge \int_{\omega\times Z_f} { f}'\cdot (  \varphi'-\ep^{-1}\widehat u_\ep')\,dx'dz+O_\varepsilon,
\end{array}
\end{equation}
where $|O_\ep|\leq C\ep^\ell$.  Passing to the limit  in (\ref{Proof_1_super_eq}) by using convergences (\ref{conv_vel_tilde_r_eq}) and (\ref{conv_vel_hat_r_eq}) when $\ep$ tends to zero, we deduce that
\begin{equation}\label{Proof_3_super_eq}\begin{array}{l}
\displaystyle
\medskip
 (\eta_0-\eta_\infty)\int_{\omega\times Z_f}(1+\lambda| \partial_{z_3}[ \varphi']|^2)^{{r\over 2}-1}\partial_{z_3}[ \varphi'] :\partial_{z_3}[ \varphi'- \widehat u']dx'dz\\
\noame
\medskip
\displaystyle+  \eta_\infty\int_{\omega\times Z_f} \partial_{z_3}[ \varphi']: \partial_{z_3}[  \varphi'- \widehat u']\,dx'dz\\
\medskip
\displaystyle+ \int_{\omega\times Z_f} \widehat p\,{\rm div}_{x'}( \varphi'- \widehat u')\,dx'dz  \ge \int_{\omega\times Z_f} { f}'\cdot (  \varphi'- \widehat u')\,dx'dz.
\end{array}
\end{equation}
We remark that the rest of the proof to derive (\ref{limit_model_gamma_eq_1}) is the same as the proof of Theorem \ref{thm_pseudoplastic} in case $1<r<2$ and $\gamma=1$, taking into account that here the exponent of the Sobolev space is $q=r$ (i.e., $(\widehat u', \widehat \pi, \widetilde p)\in V_{z_3,\#}^r(\omega\times Z_f)^2\times L^{r'}(\omega;L^{r'}_{0,\#}(Z'_f))\times (L^{r'}_0(\omega)\cap W^{1,r'}(\omega))$), so we omit it.

\end{proof}

Finally, we use Theorem~\ref{thm_dilatant} to derive the expression for filtration velocity $\widetilde V$ given in Theorem~\ref{expresiones_finalesrsup2}.

\emph{Proof of Theorem~\ref{expresiones_finalesrsup2}.} The proof of cases $\gamma<1$ and $\gamma=1$ is similar to that of case $1<r<2$, so we omit it. 

The proof of case $\gamma>1$ is similar to the proof of Theorem 3.1 (case $a_\ep\geq \ep$) in~\cite{Anguiano_SG}, so we only provide some details. By means of identification 
$$\widehat u(x',z)=\widehat w_{f'(x')-\nabla_{x'}\widetilde p(x')},\quad \widehat \pi(x',z)=\widehat q_{f'(x')-\nabla_{x'}\widetilde p(x')},$$
with $\widehat w_{\delta'}$, for every $\delta'\in\mathbb{R}^2$, the unique solution in $W^{1,r}_{\#}(Z_f)^2$ of the local reduced power-law Stokes problem 
\begin{equation}\label{limit_model_local_power}
\left\{\begin{array}{rl}
\displaystyle
-2^{-{r\over 2}}(\eta_0-\eta_\infty)\lambda^{r-2\over 2}\partial_{z_3}(|\partial_{z_3} \widehat {  w}_{\delta'}|^{r-2}\partial_{z_3}\widehat  {  w}_{\delta'})   +\nabla_{z'}\widehat q_{\delta'}(z')=-\delta' &\hbox{in }  Z_f,\\
\noame
\displaystyle
{\rm div}_{z'}(\widehat {  w}_{\delta'})=0&\hbox{in }  Z_f,\\
\noame
\displaystyle \widehat { w}_{\delta'}=0&\hbox{on } \widehat\Gamma_0\cup\widehat\Gamma_1,
\\
\noame
\displaystyle \widehat { w}_{\delta'}=0&\hbox{in }  T,\\
\noame
\displaystyle \widehat q_{\delta'}\in L^{r'}_{0,\#}(Z'_f),
\end{array}\right.
\end{equation}
we can obtain that
\begin{equation}\label{Vpower}
\widetilde V'(x')= \mathcal{U}(f'(x')-\nabla_{x'}\widetilde p(x')),\quad\widetilde V_3(x')=0,\quad \hbox{in }\omega,
\end{equation}
where the permeability function $\mathcal{U}:\mathbb{R}^2\to \mathbb{R}^2$ is defined by
\begin{equation}\label{permeabilitypower}
\mathcal{U}(\delta')=\int_{Z_f}\widehat w_{\delta'}(z)\,dz,\quad \forall\delta'\in\mathbb{R}^2.
\end{equation}

However, because problem (\ref{limit_model_local_power}) can be solved, we can provide a more explicit expression for the filtration velocity. Using~\cite[Proposition 3.4]{Tapiero}, we obtain: 
$$\widehat w_{\delta'}(z)=-{2^{r'\over 2}\over \lambda^2 (\eta_0-\eta_\infty)^{r'-1} r'}\left({1\over 2^{r'}}-\left|{1\over 2}-z_3\right|^{r'}\right)\left|\delta'+\nabla_{z'}\widehat q_{\delta'}\right|^{r'-2}(\delta'+\nabla_{z'}\widehat q_{\delta'}),
$$
A simple calculation shows  
\begin{equation}\label{intpower}\int_0^1\widehat w_{\delta'}(z)dz_3=-{1\over \lambda^2 (\eta_\infty-\eta_0)^{r'-1}2^{r'\over 2}( r'+1)}|\delta'+\nabla_{z'}\widehat q_{\delta'}|^{r'-2}(\delta'+\nabla_{z'}\widehat q_{\delta'}), 
\end{equation}
and so, from the condition ${\rm div}_{z'}(\int_0^1\widehat w_{\delta'}\,dz_3)=0$ in $Z_f'$,  we deduce that $\widehat q_{\delta'}$, $\delta'\in\mathbb{R}^2$ is the solution of (\ref{Reynolds_power}). 

From expressions (\ref{intpower}),  (\ref{Vpower}) and (\ref{permeabilitypower}), we deduce expression (\ref{FiltrationVelcovcity0r}) for $\widetilde V'$ and expression~(\ref{permeabilityA0r}) for $\mathcal{U}$, which is monotone and coercive (see Remark 3.2. in \cite{Anguiano_SG}). Finally, from the divergence conditions in (\ref{limit_model_gamma_sup_1_r}) and the expression for $\widetilde V'$, we deduce equation~(\ref{Reynolds_equation0r2}) for the pressure.

\qed

\section*{Acknowledgments}
The authors would like to thank Guy Bayada for suggesting this study on very thin porous media. M.~Bonnivard was partially supported by the ANR Project Stoiques (ANR-24-CE40-2216).  M.~Anguiano belongs to the ``Mathematical Analysis" Research Group (FQM104) at Universidad de Sevilla.

\section*{Conflict of interest declaration}
The authors declare no conflicts of interest.
 

\end{document}